\documentclass{compositio}
\usepackage{amsmath}
\usepackage{latexsym, mathrsfs}
\usepackage{hyperref}
\usepackage{relsize}
\usepackage[bottom]{footmisc}
\usepackage{color}
\usepackage{graphics,graphicx}
\usepackage{tikz-cd}
\usepackage{multirow}
\usetikzlibrary{positioning}
\usepackage{fancyhdr}
\usetikzlibrary{matrix,arrows,decorations.pathmorphing}
\theoremstyle{plain}
\newtheorem{s}{Section}[section]
\newtheorem{corollary}[s]{Corollary}

\newtheorem{theorem}[s]{Theorem}
\newtheorem{lemma}[s]{Lemma}
\newtheorem{proposition}[s]{Proposition}
\theoremstyle{definition}

\newtheorem{definition}[s]{Definition}
\theoremstyle{remark}
\newtheorem{remark}[s]{Remark}

\usepackage{lipsum}

\begin{document}
\title{2-Selmer groups of even hyperelliptic curves over function fields}
\author{Dao Van Thinh}
\date{\today}
\thispagestyle{empty}

\begin{abstract}
In this paper, we are going to compute the average size of 2-Selmer groups of families of even hyperelliptic curves over function fields. The result will be obtained by a geometric method which is based on a Vinberg's representation of the group $G=\text{PSO}(2n+2)$ and a Hitchin fibration. Consistent with the result over $\mathbb{Q}$ of Arul Shankar and Xiaoheng Wang [Rational points on hyperelliptic curves having a marked non-Weierstrass point, in \textit{Compos. Math.}, 154(1):188-222, 2018], we provide an upper bound and a lower bound of the average. However, if we restrict to the family of transversal hyperelliptic curves, we obtain precisely number 6. 
\end{abstract}
\maketitle
\tableofcontents
\section*{Introduction}
Compute the average size of Selmer groups of a family of abelian varieties has been a widespread problem nowadays. Most of the cases are considered over $\mathbb{Q},$ for instance, in \cite{BS13a}, \cite{BS13b}, \cite{BS13c}, and \cite{BS13d}, they proved that the average size of $n-$Selmer groups of elliptic curves is the sum of positive divisors of $n$ (n<6). In the higher genus situation, Bhargava and Gross show in \cite{BG13} that the average size of 2-Selmer groups of hyperelliptic curves over $\mathbb{Q}$ having a rational Weierstrass point is 3. The other interesting hyperelliptic curve is the one with a marked non-Weierstrass point, which is called even hyperelliptic curve due to its Weierstrass form. And we have the following result of A.Shankar and X. Wang:
\begin{theorem}(see \cite{SW13}) \label{theorem 1}
When all hyperelliptic curves of ?xed genus $n>2$ over $\mathbb{Q}$ having a marked rational non-Weierstrass point are ordered by height, the average size of the 2-Selmer groups of their Jacobians is at most 6.
\end{theorem}

The purpose of this paper is to complete the picture by solving the same problem in the function fields setting. To state the theorem, we first recall some notations. Let $C$ be a smooth, projective curve over the finite field $\mathbb{F}_q$ and $K$ be its function field. An even hyperelliptic curve over $K$ is given by an affine equation:
$$y^2=x^{2n+2} +c_2x^{2n}+ \cdots + c_{2n+2},$$where $c_i \in K$ for all $K$. From this equation, we can find a line bundle $\mathcal{F}_H$ over $C$ such that: $c_i \in H^0(C, \mathcal{F}_H^{\otimes i})$ for all $2 \leq i \leq 2n+2,$ and $\mathcal{F}_H$ is the "minimal" one satisfying that condition. The minimal data helps us to define the minimal integral model $\mathcal{H} \rightarrow C$ of $H,$ and we set $ht(H)$ the height of $H$ to be the degree of $\mathcal{F}_H.$ Let $\mathcal{A}_{\leq d}$ be the space of all even hyperelliptic curves over $K$ whose associated line bundles have degree less than or equal to $d$. We also denote by $\mathcal{A}^{trans}_{\leq d}$ the set of all transversal even hyperelliptic curves over $K$ whose associated line bundles have degree less than or equal to $d$ (see \ref{Definition of transversal} for the definition of transversal curves).
\begin{theorem}
	\label{main theorem 1}
	Assume that $q>4^{1/2(n+1)(2n+1)}$ if $n \geq 7,$ and $q>4^{4.(2n+1)}$ if $n<7.$ Then we have the following limits
	$$ \limsup_{d \rightarrow \infty} \frac{\sum_{H \in \mathcal{A}_{\leq d}}\frac{|Sel_2(H)|}{|\text{Aut}(H,\infty)|}}{\sum_{H \in \mathcal{A}_{\leq d}}\frac{1}{|\text{Aut}(H,\infty)|}} \leq 4. \prod_{v \in |C|} (1+|k_v|^{-n-1}) + 2 +f(q),
	$$
		$$ \liminf_{d \rightarrow \infty} \frac{\sum_{H \in \mathcal{A}_{\leq d}}\frac{|Sel_2(H)|}{|\text{Aut}(H,\infty)|}}{\sum_{H \in \mathcal{A}_{\leq d}}\frac{1}{|\text{Aut}(H,\infty)|}} \geq 4. g(q) + 2,$$
		where $f(q), g(q)$  are functions of $q$ satisfying that $\lim_{q \rightarrow \infty} f(q) =0,$ and $\lim_{q \rightarrow \infty} g(q) =1.$
\end{theorem}
Let $\mathcal{A}^{\text{ss}}$ be the family of curves whose minimal integral models is semi-stable and have irreducible fibers. Then
\begin{theorem}
	Assume that $q>4^{1/2(n+1)(2n+1)}$ if $n \geq 7,$ and $q>4^{4.(2n+1)}$ if $n<7.$ Then we have the following limits
	$$ \limsup_{d \rightarrow \infty} \frac{\sum_{H \in \mathcal{A}^{\text{ss}}_{\leq d}}\frac{|Sel_2(H)|}{|\text{Aut}(H,\infty)|}}{\sum_{H \in \mathcal{A}^{\text{ss}}_{\leq d}}\frac{1}{|\text{Aut}(H,\infty)|}} \leq 6 +f(q).
	$$
\end{theorem}
For the transversal family, we have an exact limit:
\begin{theorem}
	\label{main theorem 2}
	$$ \lim_{d \rightarrow \infty} \frac{\sum_{H \in \mathcal{A}^{trans}_{\leq d}}\frac{|Sel_2(H)|}{|\text{Aut}(H,\infty)|}}{\sum_{H \in \mathcal{A}^{trans}_{\leq d}}\frac{1}{|\text{Aut}(H,\infty)|}} = 6.
	$$
\end{theorem}
Now I want to highlight some new contributions in this paper: 
\begin{itemize}
	\item In this paper, we consider three different families of curves: general, semi-stable, and transversal families. For the first two cases, we only obtain an upper bound and a lower bound for the average size. Those numbers are functions of $q,$ and they tend to 6 when $q$ goes to infinity. More interestingly, we have an exact number for the average size of Selmer groups of transversal curves (compare to Theorem \ref{theorem 1}). The main reason is the results proved in Proposition \ref{selmer group and the first coho} saying that the size of the 2-Selmer group of $H$ is equal to the number of $J_{\mathcal{H}}[2]-$torsors over $C.$ 
	\item The connection between the type $A_{2n+1}$ Vinberg representation $(V,G)$ and hyperelliptic curves was studied before. In \cite{SW13}, they use that relation to compare the size of 2-Selmer groups and the number of certain rational orbits of $(V,G)$. In this paper, we consider the regular locus of the representation instead of the non-zero discriminant ones. The reason is that the non-regular part is of codimension 2 whether the singular part (of discriminant zero) is of codimension 1, and some of our geometric techniques only work under codimension 2 condition (for example see Proposition \ref{regular locus in general case} or \ref{regular locus in transversal case}). That is the main difference between our method to investigate the average size problem in the function fields setting and the geometry-of-number method (see \cite{SW13}) in the rational field setting. 
	\item To investigate the transversal family of curves, we need to find its density. At this place, some results in \cite{Poo03} will give us that density although we will provide a simple explanation of why that density is non-zero. Significantly, we can compute the density of monic polynomial of squarefree discriminants over $\mathbb{F}_q(t)$ (c.f \cite{bhargavasquarefree}). Furthermore, we develop the method in \cite{HLN14} to calculate the mass of semi-stable hyperelliptic curves (see \ref{regular locus in semi-stable case}). 
	\item The central part of this paper is the counting section part where we need to use the parabolic canonical reduction theory of $G-$bundles to break our bundles into "semi-stable" pieces. Since $\text{PSO}(2n+2)-$bundles are hard for the author to describe, we consider the group $\text{GSO}(2n+2)$ instead (see Section \ref{canonical reduction theory}). As a result, the induced action of $GS)(2n+2)$ on $V$ provides a semi-stable filtration of the vector bundle $V(\mathcal{E})$ associated with any $\text{GSO}(2n+2)-$bundle $\mathcal{E}$ (see \ref{V(E)}).
	\item In \cite{SW13}, after having an upper bound of the average size of 2-Selmer groups, they use Chabauty's method to study rational points on hyperelliptic curves. Hence, we may ask for a similar result in the function fields situation. We hope to work on it in the future.
\end{itemize}

The paper is organized as follows: in $\S\ref{section 1}$, we introduce the Weierstrass equation associated with a given hyperelliptic curve with a marked non-Weierstrass point. From that, we can order the family of even hyperelliptic curves by defining the heights. One of the key results in this section is Proposition \ref{selmer group and the first coho} which compares the size of 2-Selmer group of $H,$ for any hyperelliptic curve $H,$ and the number of torsors over $C$ of the 2-torsion subgroup of $J_{\mathcal{H}}.$ In $\S\ref{section 2},$ we create a Hitchin fibration from $\mathcal{M}$ to the stack of hyperelliptic curves such that: the size of any fiber $\mathcal{M}_H$ over $H$ is equal to the number of $J_{\mathcal{H}}[2]-$torsors. The study of a Vinberg-Levi's representation $(V,G)$ of type $A_{2n+1}$ and its relation to hyperelliptic curves (see Corollary \ref{jacobian and orbits}) provides a new way to count points on the stack $\mathcal{M}:$ we count global sections of some vector bundles satisfying some conditions which depend on the kind of families of curves we want to take the average over. Our final task is to counting regular sections, and it takes most parts of $\S\ref{section 3}$ and $\S\ref{section 4}.$ More precisely, in $\S\ref{section 3},$ we compute the density of regular locus in three cases: general, semi-stable, and transversal cases. Those numbers allow us to forget the regularity of sections. Combine with the canonical reduction theory of principal bundles; we can apply Riemann-Roch theorem to estimate the number of global sections. We also divide the set of $G-$bundles into small subsets based on their canonical reductions to parabolic subgroups of $G.$ The contributions of each above parts to the average can be seen in $\S\ref{section 4}$ (see Table \ref{table 1} for the summary). Finally, we prove our main theorems in $\S\ref{section 5}$ and also explain why the error term $f(q)$ does not appear in the transversal case. 
\subsection*{Notations and assumptions}
 $k=\mathbb{F}_q$ with char$(k)=p$, C is a smooth, complete, geometrically connected curve over $k$ such that $C(k) \neq \emptyset$; $K = k(C)$ the ?eld of rational functions on $C.$

Through this paper, we always assume that $p>4n$. This assumption is due to that we are going to use the canonical reduction theory in positive characteristic. Especially, for $G=\text{PSO}(2n+2)$ and $V=Sym^2_0(2n+2)$, some of the results in \cite{BH04} are valid if  the representation $(V,G)$ is of low height, i.e., $p>4n.$ Moreover, if $p>4n$ then $p$ is good for the group $G$; thus, most of the results in Vinberg theory of $\theta-$groups in characteristic zero are valid in characteristic $p$ (see \cite{Lev08} and \cite{Lev09}). 
\section{Even hyperelliptic curves over function fields} \label{section 1}
\subsection{Heights of even hyperelliptic curves}
 Let $H$ be a hyperelliptic curve over $K$ with a marked non-Weierstrass $K-$rational point $\infty$. Let $\infty'$ be the conjugate of $\infty$ under the hyperelliptic involution, then by considering the set of global sections of $\mathcal{O}_H(\infty + \infty')$, we can show that the curve $H$ is defined by an affine Weierstrass equation:
$$y^2 = x^{2n+2} +c_2x^{2n}+ \cdots + c_{2n+2},$$where $c_i \in K$ for all $K$. Because of that, such hyperelliptic curves are called to be even. Moreover, the tuple $(c_2,c_3,\dots,c_{2n+2})$ is unique up to the following identification:
$$(c_2,c_3,\dots,c_{2n+2}) \equiv (\lambda^2.c_2, \lambda^3c_3,\dots, \lambda^{2n+2}.c_{2n+2}) \hspace{2cm} \lambda \in K^{\times}.$$ 
Given the data $(c_2,c_3,\dots,c_{2n+2})$, we define the minimal integral model of $H$ as follows: for each point $v \in |C|$, we can choose an integer $n_v$ which is the smallest integer satisfying that: the tuple $(\varpi_v^{2n_v}c_2, \varpi_v^{3n_v}c_3, \cdots, \varpi_v^{(2n+2)n_v}c_{2n+2})$ has coordinates in $\mathcal{O}_{K_v}$. Given $(n_v)_{v \in |C|}$, we define the invertible sheaf $\mathcal{F}_H \subset K$ whose sections over a Zariski open $U \subset C$ are given by $$\mathcal{F}_H(U) = K \cap \big(\prod_{v \in U}\varpi_v^{-n_v}\mathcal{O}_{K_v} \big).$$
Then it is easy to deduce that $c_i \in H^0(C, \mathcal{F}_H^{\otimes i})$ for all $2 \leq i \leq 2n+2$. Furthermore, the stratum $(\mathcal{F}_H, \underline{c})$ is minimal in the sense that there is no proper subsheaf $\mathcal{M}$ of $\mathcal{F}_H$ such that $c_i \in H^0(C, \mathcal{M}^{\otimes i})$ for all $i$. Conversely, given a minimal stratum $(\mathcal{F}, \underline{c})$ satisfying that $\Delta(\underline{c}) \neq 0$, we consider a subscheme of the weighted projective space $\mathbb{P}_{(1,n+1,1)}( \mathcal{F} \oplus \mathcal{F}^{n+1} \oplus \mathcal{O}_C)$ that is defined by 
$$Y^2 = X^{2n+2} +c_2Z^2X^{2n} +\dots + c_{2n+1}Z^{2n+1}X + c_{2n+2}Z^{2n+2}.$$ This is a flat family of curves $\mathcal{H} \rightarrow C,$ and the generic fiber $\mathcal{H}_K$ is a hyperelliptic curve over $K(C)$ with a marked rational non-Weierstrass point $[1:1:0]$ at infinity. Furthermore, the associated minimal data of $H$ is precisely $(\mathcal{F}, \underline{c})$. Hence we have just shown the surjectivity of the following map $\phi_{\mathcal{F}}$ with a given line bundle $\mathcal{F}$ over $C$:
$$\phi_{\mathcal{F}}: \{ \text{minimal tuples} \hspace{0.2cm}(\mathcal{F}, \underline{c}) \} \rightarrow \{ \text{Hyperelliptic curves $(H, \infty)$ such that $\mathcal{F}_H \cong \mathcal{F}$} \}.$$
Moreover, the sizes of fibers of $\phi_{\mathcal{F}}$ can be calculated as follows:
\begin{proposition} \label{minimal data}
Given a line bundle $\mathcal{F}$ over $C$, the map $\phi_{\mathcal{F}}$ defined as above is surjective, and the preimage of $(H,\infty)$ is of size $\frac{|\mathbb{F}_q^{\times}|}{|\text{Aut}(H,\infty)|}$, here $\text{Aut}(H,\infty)$ denotes the subset of all elements in $\text{Aut}(H)$ which preserve the marked point $\infty$.
\end{proposition}
\begin{proof}
Suppose that $(H,\infty)$ is a hyperelliptic curve with the associated minimal data $(\mathcal{F}, \underline{c}).$ Since we fix the line bundle $\mathcal{F}$, the tuple of sections $\underline{c}$ is well-defined up to the following identification: 
$$\underline{c} \equiv \lambda. \underline{c} = (\lambda^2c_2, \dots, \lambda^{2n+1}c_{2n+1}, \lambda^{2n+2}c_{2n+2}), \hspace{1cm} \lambda \in \mathbb{F}_q^{\times}.$$
In other words, there is a transitive action of $\mathbb{F}_q$ on the fiber $\phi_{\mathcal{F}}^{-1}(H)$. Furthermore, the stabilizer of any element in $\phi_{\mathcal{F}}^{-1}(H)$ is precisely $\text{Aut}(H,\infty)$. Hence, the size of $\phi_{\mathcal{F}}^{-1}(H,\infty)$ is $\frac{|\mathbb{F}_q^{\times}|}{|\text{Aut}(H,\infty)|}$. 
\end{proof}

\begin{definition}{\textbf{(Height of hyperelliptic curve)}} The height of the hyperelliptic curve $(H,\infty)$ is defined to be the degree of the associated line bundle $\mathcal{F}_H$. 
\end{definition}
The transversality of even hyperelliptic curves also can be defined by using the above integral minimal model as follows:

\begin{definition} \label{Definition of transversal}
An even hyperelliptic curve $H$ with an associated minimal data $(\mathcal{F}_H, \underline{c})$ is called to be transversal if the discriminant $\Delta(\underline{c}) \in H^0(C, \mathcal{F}_H^{\otimes (2n+1)(2n+2)})$ is square-free. 
\end{definition}
\subsection{2-Selmer group and the first cohomology group}
For each hyperelliptic curve $H$ over $K(C)$, we associate to it a flat family of curves $$h:\mathcal{H} \rightarrow C$$ with reduced geometric fibers. Moreover, any irreducible component of a reducible curve in that family is geometrically integral. Hence, by \cite{Ray90} Theorem 8.2.2, the relative Jacobian functor, which classifies invertible sheaves of degree 0, is represented by a group scheme locally of finite type over $C$, denoted by $J_{\mathcal{H}}.$ Furthermore, the Jacobian $J_H$ of $H$ can be identified as the generic fiber of $J_{\mathcal{H}}.$ Recall that we demote by $Sel_2(H)$ the $2-Selmer$ group of the Jacobian $J_{H}$. We now are able to relate $|Sel_2(H)|$ and the number of $J_{\mathcal{H}}-$torsors over $C$ as follows:
\begin{proposition}\label{selmer group and the first coho}
 Let $H$ be a hyperelliptic curve over the function field $K(C).$ If $H$ is transversal (see Definition \ref{Definition of transversal}), then we have equality:
$$|Sel_2(H)| = |H^1(C, J_{\mathcal{H}}[2]|.$$
In general, we have an inequality:
$$|Sel_2(H)| \leq |H^0(K, J_H[2])| . |H^1(C, J_{\mathcal{H}}[2])|.$$
\end{proposition}
\begin{proof}
For the transversal case, see [\cite{DVT17}, Proposition 2.20]. Now we consider the general case. Let $\mathcal{J}_{H}$ over $C$ be the N\'{e}ron model of $J_{H}$, and $\mathcal{J}^0_H$ be its identity connected component, i.e., over each point $v \in |C|,$ $\mathcal{J}^0_{H,v}$ is the identity component of $\mathcal{J}_{H,v}.$ Follow [\cite{aaron}, Proposition 3.24] (this result is for elliptic curves case, but its proof can be extended for the case of Jacobians of hyperelliptic curves), we are able to prove that
\begin{equation} 
|Sel_2(H)| \leq |H^0(K, J_H[2])| . \frac{|H^1(C, \mathcal{J}^0_{H}[2])|}{H^0(C, \mathcal{J}^0_H[2])}. \label{1}
\end{equation} All we need to do now is to compare $|H^1(C, \mathcal{J}^0_{H}[2])|$ and $|H^1(C, J_{\mathcal{H}}[2])|.$ By using the N\'eron mapping property, we first obtain a canonical morphism $j: J_{\mathcal{H}} \rightarrow \mathcal{J}_H$
of group schemes over $C$. Since $J_{\mathcal{H},v}$ is connected for any $v \in |C|,$ the map $j$ factors through $\mathcal{J}^0_H.$ We also denote the induced morphism from $J_{\mathcal{H}}[2]$ to $\mathcal{J}^0_H[2]$ by $j$. In the abelian category of sheaves of finite abelian groups over $C$, we have the following exact sequences:
\begin{align}
	0 \rightarrow \text{Ker}(j) \rightarrow J_{\mathcal{H}}[2] \rightarrow \text{Im}(j) \rightarrow 0 \label{exact sequence 1}\\
	0 \rightarrow \text{Im}(j) \rightarrow \mathcal{J}^0_H[2] \rightarrow \text{Coker}(j) \rightarrow 0. \label{exact sequence 2} 
 \end{align} 
 Since there are only finitely many points $v$ of $C$ such that $\mathcal{H}_v$ is not smooth (the number of those points is bounded by the degree of $\Delta(H)$), we imply that $J_{\mathcal{H}}$ and $\mathcal{J}^0_H$ are agreed over $C$ except at a finite number of points. This deduce that $\text{Ker}(j)$ and $\text{Coker}(j)$ are skyscraper sheaves supported at finitely many points of $C$. Thus, we have the following qualities:
 $$|H^0(C,\text{Ker}(j))| = |H^1(C, \text{Ker}(j))|; \,\, |H^0(C, \text{Coker}(j))| = |H^1(C, \text{Coker}(j))|; $$
 $$\,\, H^2(C,\text{Ker}(j)) = H^2(C, \text{Coker}(j)) = 0.$$
 By considering two long exact sequences of cohomology groups associated to (\ref{exact sequence 1}) and (\ref{exact sequence 2}), we find
 \begin{align*}
  |H^0(\text{Ker}(j))|. |H^0(\text{Im}(j))| . |H^1(J_{\mathcal{H}}[2])| &= |H^0(J_{\mathcal{H}[2]})|. |H^1(\text{Ker}(j))|. |H^1(\text{Im}(j))|  \\
 \Rightarrow  |H^0(\text{Im}(j))| . |H^1(J_{\mathcal{H}}[2])| &= |H^0(J_{\mathcal{H}[2]})| . |H^1(\text{Im}(j))|,
 \end{align*}
and 
\begin{align*}
 |H^0(\text{Im}(j))|. |H^0(\text{Coker}(j))| . |H^1(\mathcal{J}^0_{H}[2])| &\leq  |H^0(\mathcal{J}^0_H[2])|. |H^1(\text{Im}(j))|. |H^1(\text{Coker}(j))|  \\
\Rightarrow  |H^0(\text{Im}(j))|. |H^1(\mathcal{J}^0_{H}[2])| &\leq  |H^0(\mathcal{J}^0_H[2])|. |H^1(\text{Im}(j))|. 
\end{align*}
We imply that
\begin{equation} \label{4}
	\frac{|H^1(\mathcal{J}^0_H[2])|}{|H^0(\mathcal{J}^0_H[2])|} \leq \frac{|H^1(J_{\mathcal{H}}[2])|}{|H^0(J_{\mathcal{H}}[2])|}
\end{equation}
Combine (\ref{1}) and (\ref{4}), we have completed the proof. 
\end{proof}
\begin{remark}
	\begin{itemize}
		\item[i)]  For any curves $H$ whose minimal integral model $\mathcal{H} \rightarrow C$ is regular with geometrically irreducible fibers, by [\cite{Ray90}, Proposition 9.5.1], $J_\mathcal{H}$ coincides with the N\'eron model of $J_H.$ Then we can also show that $|Sel_2(H)| = |H^1(C, J_{\mathcal{H}}[2]|.$ Although the family of curves with regular Weierstrass model is quite big (any transversal curves belong to this family), it is hard to estimate the size of this family in contrast to the transversal case (see Section \ref{section 3}).
		\item [ii)] If $\mathcal{H} \rightarrow C$ is semi-stable (see Section \ref{semi-stable section}), we deduce that $J_\mathcal{H} = \mathcal{J}^0_H$ since $\mathcal{H}$ has rational singularities. By looking at the proof of the above proposition, we realize that the chance to obtain an exact average size in the semi-stable case is higher than the general case. We hope to consider that problem in the future.
		\end{itemize}
\end{remark}
From Proposition \ref{selmer group and the first coho}, we see that in general case, $|Sel_2(H)|$ is bounded by $H^1(C, J_{\mathcal{H}}[2])$ if $H^0(K, J_H[2]) = 0.$ The next proposition will help us to ignore the case $H^0(K, J_H[2]) \neq 0$ if the size of our base field $\mathbb{F}_q$ is large enough.
\begin{proposition}
	Assume that $q>4^{1/2(n+1)(2n+1)}$ if $n \geq 7,$ or $q>4^{4.(2n+1)}$ if $n<7.$ Then the contribution of the case $H^0(K, J_H[2]) \neq 0$ to the average is zero.
\end{proposition}
\begin{proof}
Given a hyperelliptic curve $H: y^2 = f(x)$ with the associated minimal data $(\mathcal{F}_{H}, \underline{c})$, assume that $J_H[2](K) \neq 0.$ Note that we have a correspondence between elements in $J_H[2](K)$ and factorizations of $f(x)$ into $g(x).h(x),$ where $g(x)$ and $h(x)$ are in $K[x]$ of even degree, or $g(x) = \overline{h}(x)$ in some quadratic extensions of $K$. Thus, we have two cases as follows:
\begin{itemize}
\item[Case 1:] There exists an even factorization $g(x).h(x)$ in $K[x]$ of $f(x).$ Recall that $f(x) = x^{2n+2} + c_2x^{2n} + \dots + c_{2n+1}x + c_{2n+2},$ where $c_i \in H^0(C, \mathcal{F}_H^{\otimes i})$ for all $i)$ More precisely, for any $v \in |C|,$ we have that $\text{val}_v(c_i) + i.deg_v(\mathcal{F}_H) \geq 0$ for all $i)$ If we write 
\begin{align*}
	g(x) = x^{n_1} + a_1x^{n_1-1} + \dots + a_{n_1}, \\
	h(x)= x^{n_2} + b_1x^{n_2-1} + \dots + b_{n_2},
\end{align*}
then it can be shown that $a_i \in H^0(C, \mathcal{F}_H^{\otimes i})$ for all $1 \leq i \leq n_1,$ and $b_j \in H^0(C, \mathcal{F}_H^{\otimes j})$ for all $1 \leq j \leq n_2.$ In fact, it is equivalent to show that for any $v \in |C|,$ $\text{val}_v(a_i) + i.deg_v(\mathcal{F}_H) \geq 0$
 and $\text{val}_v(b_j) + j.deg_v(\mathcal{F}_H) \geq 0$ for all $i, j.$ By multiplying $f(x)$ by $\varpi_v^{(2n+2).deg_v(\mathcal{F}_H)}$ and changing $x$ to $x. \varpi_v^{deg_v(\mathcal{F}_H)},$ we may assume that $c_i \in \mathcal{O}_{K_v}$ for all $2 \leq i \leq 2n+2,$ and we need to show that $g(x)$ and $ h(x)$ are in  $\mathcal{O}_{K_v}[x].$ But this is coming from the fact that $\mathcal{O}_{K_v}$ is a unique factorization domain. Additionally, it is easy to see that $a_1=-b_1$ since the coefficient of $x^{2n+1}$ in $f(x)$ is zero. To sum up, by using the Riemann-Roch theorem, we may bound above the number of hyperelliptic curves in case 1 with fixed minimal line bundle $\mathcal{F}$ (assume that $\text{deg}(\mathcal{F})$ is large enough) by  
\begin{equation} \label{5}
	\sum_{i=1}^{n}q^{(4i^2-4i(n+1)+(n+2)(2n+1))d +(2n+1)(1-g)}.
\end{equation}
\item [Case 2:] $f(x) = h(x). \overline{h}(x)$ in a quadratic extension $K'$ of $K.$ We write $$h(x) = x^{n+1} + (a_1 + t.b_1)x^n + \dots + (a_{n+1} + t.b_{n+1}),$$where $a_i, b_i \in K$ for all $i,$ and $t \in K' \setminus K$ such that $t^2 \in K.$ For each $v \in |C|$, we have a unique extension to $K_v'$ of the valuation $\text{val}_v$ on $K_v.$ Precisely, for each element $a +t.b \in K_v',$ we set 
$$\text{val}_{K_v'}(a+t.b) := 1/2 . \text{val}_v(a^2 - t^2.b^2).$$  
Now we use a similar argument as in Step 1 to get 
\begin{align*}
	\text{val}_{K_v'}(a_i+t.b_i) + i. deg_v(\mathcal{F}_H) \geq 0, \,\, \text{for all $1\leq i \leq n+1,$}
	\end{align*} 
	 \[\Rightarrow  \left\{ \begin{array}{ll}
	\text{val}_v(a_i) = \text{val}_{K_v'}(a_i+t.b_i + a_i - t.b_i) \geq  -i. deg_v(\mathcal{F}_H) \,\, \text{for all $1\leq i \leq n+1,$} \\
	\text{val}_v(t^2.b_i^2) =2 \text{val}_{K_v'}\big((a_i+t.b_i) - (a_i - t.b_i)\big) \geq  -2i. deg_v(\mathcal{F}_H) \,\, \text{for all $1\leq i \leq n+1.$}
	\end{array} \right. \] 	
	 This implies that for fixed line bundle $\mathcal{F}_H,$ the number of choices of $a_i+t.b_i$ is bounded above by $2|H^0(C, \mathcal{F}_H^{\otimes i})| . |H^0(C, \mathcal{F}_H^{\otimes 2i})|$ for any $1 \leq i \leq n+1.$ Moreover, we easily deduce from the vanishing of $x^{2n+1}$ in $f(x)$ that $a_1= 0.$ Hence, if the degree of $\mathcal{F}_H$ is large enough, the number of hyperelliptic curves, in this case, is bounded above by
	  \begin{equation} 
	  \label{6} 
	  2^{n+1}q^{3\sum_{i=1}^{n+1}i -1} = 2^{n+1} q^{\frac{1}{2}(3n^2+9n+4)}.
	  \end{equation}
\end{itemize}
From (\ref{5}) and (\ref{6}), we conclude that the number of hyperelliptic curves satisfying that their associated line bundle is $\mathcal{F}$ of large degree $d$ and their Jacobians possess nontrivial 2-torsion points over $K$, is bounded above by
\[ \left\{ 
\begin{array}{ll}
m. q^{(n+2)(2n+1)d - 4nd} \,\,\,\,\,\,\,\,\,\,\,\,\,\,\,\,\,\,\,\text{if $n \geq 7$} \\
m.q^{(n+2)(2n+1)d - 1/2(n^2+n)d} \,\, \text{otherwise,}
\end{array}
\right.\]
where $m$ is a number that is independent to $d.$ Now, we use the same argument as in [\cite{DVT17}, Lemma 2.12] to obtain an upper bound for the average size of 2-Selmer groups in case $J[2](K) \neq 0$:
\[ \left\{ 
\begin{array}{ll}
\frac{m'.q^{(n+2)(2n+1)d - 4nd}. 4^{n.(2n+1)(2n+2)d}}{q^{(n+2)(2n+1)d}} = m'. \frac{4^{n(2n+1)(2n+2)d}}{q^{4nd}} \,\,\,\,\,\,\,\,\,\,\,\,\,\,\,\,\,\,\text{if $n \geq 7$} \\
\frac{m'.q^{(n+2)(2n+1)d - 1/2(n^2+n)d}. 4^{n.(2n+1)(2n+2)d}}{q^{(n+2)(2n+1)d}} = m'. \frac{4^{n(2n+1)(2n+2)d}}{q^{1/2(n^2+n)d}} \,\,\, \text{if $n<7$.}
\end{array}
\right.\]So when $d \rightarrow \infty,$ the above numbers tend to zero under assumptions that $q>4^{1/2(n+1)(2n+1)}$ if $n \geq 7,$ and $q>4^{4.(2n+1)}$ if $n<7.$
\end{proof}
\section{Vinberg-Levi's representation and a connection to hyperelliptic curves} \label{section 2}
\subsection{Vinberg-Levi's representation}
\label{Vinberg representation}
 Let $(U,Q)$ be the split quadratic space over $k$ of dimension $2n+2$ and discriminant $1$. Then for any linear operator $T: U \rightarrow U,$ we defined its adjoint $T^*$ by the following equation:
$$\langle Tv,w \rangle_Q = \langle v, T^*w\rangle_Q, \hspace{1cm} \forall v,w \in U.$$where $\langle v,w \rangle_Q=Q(v+w) - Q(v) -Q(w)$ indicates the bilinear form associated with $Q.$ The Vinberg-Levi's representation we are going to study is the conjugate action of 
$$G:= \text{PSO}(U) = \{ g \in \text{GL}(U) | gg^*=I, \text{det}(g)=1 \}/\mu_2$$on
$$V=\{ T: U \rightarrow U | T=T^*, trace(T)=0 \}.$$
The above representation is of type $A_{2n+1}$ in Vinberg-Levi's theory, and we know that the GIT quotient $V//G$ is isomorphic to $S = \text{Spec}(k[c_2,c_3,\dots, c_{2n+2}]),$ where for each $T \in V$, $c_i(T)$ is the coefficients of the characteristic polynomial of $T$:
$$f_T(x) = x^{2n+2} + c_2(T) x^{2n} + \cdots + c_{2n+1}(T)x +c_{2n+2}(T).$$
We indicate the projection map by $\pi : V \rightarrow S$. In the next subsection, we are going to introduce some sections of this map which is essential to define regular locus. 
\subsection{Kostant sections and regular locus}
We are interested in elements in $V$ whose stabilizers in $G$ are of finite order, and let denote 
$$V^{\text{reg}}(\overline{k}) = \{ T \in V(\overline{k}) | |\text{Stab}_{G(\overline{k})}(T)| \text{ is finite} \}.$$
We have the following criteria for regularity:
\begin{proposition}
An element $T$ in $V(\overline{k})$ is regular if and only if the characteristic polynomial of $T$ is equal to its minimal polynomial.
\end{proposition}
Since equality of characteristic and minimal polynomials is independent to field extensions, the above Proposition helps us to define an open subscheme $V^{\text{reg}}$ of $V$ which is defined over $k$. The stabilizer of $T \in V^{\text{reg}}$ could be computed as follows:
\begin{proposition}
\label{stabilizer}
Let $F$ be an extension of $k$, and $T \in V^{\text{reg}}(F)$ be a regular element. Set $L = F[x]/(f_T(x)),$ where $f_T(x)$ is the characteristic polynomial of $T$. Then we have the following isomorphism:
$$\text{Stab}_G(T) \cong (\text{Res}_{L/F}\mu_2)_{N=1}/\mu_2.$$
\end{proposition}
\begin{proof}
(c.f. \cite{Wan1})
\end{proof}
Now, given an invariant $\underline{c} = (c_2,c_3, \dots, c_{2n+2}) \in S,$ we define an element in $V$ whose invariant is $\underline{c}$ as follows:
$$\kappa_1(\underline{c}) = A + B,$$where $A, B \in \text{Mat}_{2n+2}$; $A$ is the lower diagonal nilpotent matrix, i.e., the entries of $A$: $a_{ij}=0$ except the case that $i=j+1$; $B$ is defined as follows:
$$B = \left(\begin{array}{cc} 0_{n+1} & C \\ 0_{n+1} & 0_{n+1} \end{array} \right),$$where $0_{n+1}$ is the zero matrix of size $n+1$, and $C$ is the following triple anti-diagonal matrix:
$$C = \left(\begin{array}{cccccc} 0 & 0 & \cdots & 0 & a_{2n+1}/2 & -a_{2n+2} \\
0 & 0 & \cdots & a_{2n-1}/2 & -a_{2n} & a_{2n+1}/2 \\
0 & 0 & \cdots & -a_{2n-2} & a_{2n-1}/2 & 0 \\
\vdots &\vdots&\vdots&\vdots&\vdots&\vdots \\
a_3/2 & -a_4 & \cdots & 0 & 0& 0 \\
-a_2  & a_3/2 & \cdots & 0 & 0& 0
 \end{array} \right).$$
It is easy to check that $\kappa_1(\underline{c}) \in V$ is regular. By varying $\underline{c} \in S,$ we get a section of the invariant map $\pi: V \rightarrow S$ whose image is in the regular locus $V^{\text{reg}}$. Furthermore, the second section of $\pi$ could be defined as follows: $\kappa_2(\underline{c}) = J.\kappa_1(\underline{c}).J^*,$ where 
$$J = \left(\begin{array}{ccc}
&&1 \\
&I_{2n} & \\
1&&
\end{array}\right) \in \text{O}(U)\setminus \text{SO}(U).$$
We call two above sections $\kappa_1$ and $\kappa_2$ the Kostant sections. 
\begin{remark}
\begin{itemize}
\item[i)] By Proposition 1.51 in \cite{Wan1}, over $\overline{k}$, the action of $\text{PO}(U)$ on $V^{\text{reg}}_f:= V^{\text{reg}} \cap \pi^{-1}(f)$ is transitive for any $f \in S$. This implies that any elements of $V^{\text{reg}}_f$ are conjugate with at least one of $\kappa_1(f)$ and $\kappa_2(f)$ by an element of $G = \text{PSO}(U).$ In other words, $V^{\text{reg}}(\overline{k}) = G . \kappa_1(S) \cup G . \kappa_2(S).$

\item[ii)] In Vinberg-Levi's theory of $\theta-$groups, if we set $H= \text{SL}(U)$ and $\theta$ denotes the inverse transpose operator on $H$, then our group $G$ will be the fixed subgroup $H^{\theta=1}$ of $H$. Moreover, if we denote $Z(H)$ to be the center of $H$, and set $H_{ad}^{\theta}= \{ h \in H | \theta(h) \in Z(H).h \},$ then by [\cite{Lev08}, Theorem 0.14], $H_{ad}^{\theta}$ acts transitively on $V^{\text{reg}}_f$. It is also easy to check that $\text{SO}(U)$ is the identity connected component of $H_{ad}^{\theta}$ of order $2$. Hence, the action of $\text{SO}(U)$ on $V^{\text{reg}}_f$ has at most $2$ orbits as expected. 
\end{itemize}
\end{remark}
The finiteness of the stabilizer of regular elements implies that the action map
\begin{align*}
G \times_{S}V^{\text{reg}} & \longrightarrow  V^{\text{reg}}\times_{S} V^{\text{reg}} \\
(g,v)&\mapsto (g.v,v)
\end{align*}
is \'{e}tale, so the universal stabilizer $I$ of the action of $G$ on $V^{\text{reg}}$
$$I= (G\times_{S}V^{\text{reg}}) \times_{V^{\text{reg}}\times_{S} V^{\text{reg}}}V^{\text{reg}},$$where $V^{\text{reg}} \rightarrow V^{\text{reg}}\times V^{\text{reg}}$ is the diagonal map, is a quasi-finite \'{etale} group scheme over $V^{\text{reg}}$. One of the key results in this section is that we can descend the group scheme $I$ to a scheme over $S$. That is the content of the following proposition:

\begin{proposition} There exists a unique group scheme $I_{S}$ over $S$ equipped with an isomorphism $\pi^*I_{S} \rightarrow I$ over $V^{\text{reg}}$. This isomorphism is $G-$equivariant; thus, as a result, there is a $\mathbb{G}_m-$equivariant isomorphism of stacks $[BI_{S}]\cong [V^{\text{reg}}/G],$ where $BI_{S}$ is the relative classifying stack of $I_{S}$ over $S$.
\label{stabilizer and orbits}
\end{proposition}

\subsection{Connection to hyperelliptic curves}
For each  $\underline{c}=(c_2, c_3, \dots, c_{2n+2}) \in S(\overline{k})$, the associated polynomial
$$f_{\underline{c}}(x) = x^{2n+2} + c_2x^{2n} + \dots + c_{2n+1}x + c_{2n+2}$$defines an even hyperelliptic curve $y^2= f_{\underline{c}}(x)$ (we allow singular hyperelliptic curves).  Then by varying $\underline{c}$, we obtain a flat family of projective, connected, reduced curves. Moreover, any irreducible component of a reducible curve in that family is geometrically integral. Hence, by [\cite{Ray90}, Theorem 8.2.2], the relative Jacobian functor, which classifies invertible sheaves of degree 0, is represented by a group scheme locally of finite type over $S$, denoted by $J_S.$ Recall that we have the invariant map:
$$\pi: V^{\text{reg}} \rightarrow S,$$thus, the fiber product
$$J_{V^{\text{reg}}}:= J_S \times_S V^{\text{reg}}$$defines a group scheme over $V^{\text{reg}}$. The main result of this subsection is the following isomorphism:
\begin{proposition} \label{stabilizer and jacobian}
There exists a canonical isomorphism over $V^{\text{reg}}$ between the stabilizer scheme $I$ (defined before proposition \ref{stabilizer and orbits}) and $J_{V^{\text{reg}}}[2]$. 
\end{proposition}
\begin{proof}
For each regular element $T \in V^{\text{reg}}(\overline{k})$, we will construct a Galois invariant isomorphism from $I_T(\overline{k})$ and $J_T[2](\overline{k})$. Set $f_T(x)$ to be its characteristic polynomial. Then the stabilizer of $T$ in $G=\text{PSO}(U)$ is isomorphic to
$$(\text{Res}_{L/\overline{k}}\mu_2)_{N=1}/\mu_2,$$where $L=\overline{k}[x]/(f_T(x))$, and $N$ is the norm map (c.f. \cite{Wan1}). On the other hand, if we factorize $f_T(x)$:
$$f_T(x)= \prod_{i=1}^t (x-\alpha_i)^{m_i},$$then the subgroup of 2-torsion points of the (generalized) Jacobian $J_{T}$ is 
$$J_T[2] = \big< P_i + P_j - \infty - \infty' | \sum_{i=1}^t m_iP_i - (n+1)(\infty + \infty') =0 \big>_{1\leq i,j \leq t}\,\,\,\,,$$where $P_i$ is the Weierstrass point corresponding to the root $\alpha_i$ of $f_T(x)$, $\infty$ is the marked rational non-Weierstrass point, and $\infty'$ is conjugate to $\infty$ via the hyperelliptic involution. We have two cases as follows:

\textbf{Case 1:} $m_i$ is even for all $1\leq i \leq t$. In this case, firstly, the associated finite group $J_T[2]$ of the reducible curve $y^2=f_T(x)$ could be described as follows: if $t=1$ then it is trivial. If $t>1$, we set $J_{\overline{T}}$ to be the Jacobian of the curve $y^2 = f_T(x)/(x-\alpha_t)^{m_t}$. There is a surjective map 
$$\psi : J_T \rightarrow J_{\overline{T}}$$whose kernel is isomorphic to $\mathbb{G}_m \times \mathbb{G}_a^{m_t/2-1}$. This induces an isomorphism 
$$J_T[2](\overline{k}) \cong J_{\overline{T}}[2](\overline{k}) \times \mathbb{Z}/2$$
We keep doing this kind of reduction until there is only one factor $(x-\alpha_1)^{m_1}$ left. For each $2\leq i \leq t$, we denote $Q_i \in J_T$ to be the 2-torsion point that corresponds to  $(x-\alpha_i)^{m_i}$ as in the above reduction. Then we define a morphism:
$$\phi_T: (\text{Res}_{L/\overline{k}}\mu_2)/\mu_2 \rightarrow J_T[2](\overline{k})$$that maps the representative $(1,\dots,1,-1,1,\dots,1),$ where $-1$ is in the $i^{th}$ position, to $Q_i$. It is easy to check that $\phi_T$ is an isomorphism and it does not depend on the way we choose $Q_i$. In other words, if we fix another factor $(x-\alpha_i)^{m_i}$ instead of $(x-\alpha_1)^{m_1}$ and define an isomorphism $$\phi_T':(\text{Res}_{L/\overline{k}}\mu_2)/\mu_2 \rightarrow J_T[2](\overline{k})$$that maps the representative $(1,\dots,1,-1,1,\dots,1),$ where $-1$ is in the $j^{th}$ position, for $1\leq j \neq i \leq t$, to $Q_j$, then two morphisms $\phi_T$ and $\phi_T'$ are identical. Thus, our map is Galois invariant.

\textbf{Case 2:} $m_i$ is not even for all $1\leq i \leq t$. Without loss of generality, we assume that for $1 \leq i \leq h <t$, $m_i$ is even, and the remaining $m_j$, for $h+1 \leq j \leq t$, is odd. In this case, the group $J_T[2](\overline{k})$ can be seen as follows:
$$J_T[2](\overline{k})=\big< Q_1,\dots, Q_h, P_j-P_r | \text{div}(y) = 0 \big>_{h+1 \leq j < r \leq t} \,\, ,$$where $Q_i$ is the 2-torsion point that corresponds to $(x-\alpha_i)^{m_i}$ for $1 \leq i \leq h,$ and $P_j$, for $h+1 \leq j \leq t$, is the Weierstrass point corresponding to the root $\alpha_j$. Similar to the previous case, we make some reductions on $f_T(x)$ (to obtain $Q_i$) until there is only one even factor left. We may assume that the remaining even factor is $(x-\alpha_1)^{m_1}$ and can check later that our defining isomorphism will not depend on this choice. If we denote $J_{\overline{T}}$ be the Jacobian of the curve $$H: \,\,\,y^2= (x-\alpha_1)^{m_1} . \prod_{j=h+1}^t (x-\alpha_j)^{m_j}=f_{\overline{T}}(x),$$then we first have the following isomorphism:
$$J_T[2](\overline{k}) \cong (\mathbb{Z}/2)^{h-1} \times J_{\overline{T}}[2](\overline{k}).$$
Additionally, $J_{\overline{T}}[2](\overline{k})$ can be described as follows: we consider the curve $\overline{H}: \,\, y^2= \frac{f_{\overline{T}}(x)}{(x-\alpha_1)^{m_1}} $and the normalization map
$$\phi: \overline{H} \rightarrow H.$$Let $P$ and $Q$ be two preimages of $(\alpha_1,0) \in H$. Then the generalized Jacobian of $H$ is the quotient of the set of zero degree divisors by the principal divisors which are of the form $\text{div}(h)$, where $h$ is a rational function satisfying $h(P)=h(Q)$. We conclude that $J_{\overline{T}}[2](\overline{k})$ is canonically isomorphic to:
$$\big< P_j-P_r \big>_{h+1 \leq j <r \leq t}.$$ Note that the condition $\text{div}(y)$ does not appear here since $y(P) \neq y(Q).$ Finally, we can define a morphism
$$\psi_T: J_T[2](\overline{k}) \rightarrow (\text{Res}_{L/\overline{k}}\mu_2)_{N=1}/\mu_2$$given by
\begin{align*}
      Q_i &\mapsto (1,\dots,1,-1,1,\dots,1) \,\,\text{for} \,\,2 \leq i \leq h, \text{where $-1$ is in the $i^{th}$ position, } \\
      P_j-P_r &\mapsto (1,\dots,1,-1,1\dots,1,-1,1,\dots,1) \,\, \text{for $h+1\leq j < r \leq t,$ where $-1$ is in } \\ & \hspace{9cm}\text{the $j^{th}$ and $r^{th}$ positions.} 
\end{align*}
 We can check that $\psi_T$ is an isomorphism, and it does not depend on the way we define it by fixing $\alpha_1$. As a result, $\psi_T$ is Galois invariant.  
\end{proof}
We consider the action of $G$ on $J_{V^{\text{reg}}}= J_S \times_S V^{\text{reg}}$ that is induced from the action of $G$ on $V^{\text{reg}}$. Then we can see that the isomorphism in the above Proposition between $I_{|V^{\text{reg}}}$ and $J_{V^{\text{reg}}}[2]$ is $G-$equivariant. Thus, it produces an isomorphism from $I_S$ to $J_S[2]$. Moreover, we have a natural action of $\mathbb{G}_m$ on $J_S[2]$ as follows: for each $g \in \mathbb{G}_m$, $\underline{c} \in S$, recall that the action of $\mathbb{G}_m$ on $S$ is given by:
$$g.\underline{c}= (g^2c_2,\dots,g^{2n+1}c_{2n+1}, g^{2n+2}c_{2n+2}).$$The canonical isomorphism between two hyperelliptic curves $H_{\underline{c}}$ and $H_{g.\underline{c}}$ induces a group scheme isomorphism between their Jacobians. Hence, we have an isomorphism (denoted by $g$)
$$g: \,\,\, J_{H_{\underline{c}}}[2] \rightarrow J_{H_{g.\underline{c}}}[2]$$of finite group schemes over $k$. By varying $\underline{c} \in S$, we have an action of $g$ on $J_S[2]$. It is easy to check that the isomorphism in Proposition \ref{stabilizer and jacobian} is $\mathbb{G}_m-$equivariant. Combine with Proposition \ref{stabilizer and orbits}; we have the following corollary.
\begin{corollary}\label{jacobian and orbits}
There is a $\mathbb{G}_m-$equivariant isomorphism between the classifying stack $BJ_S[2]$ and the quotient stack $[V^{\text{reg}}/G].$
\end{corollary}
\subsection{Restatements of main theorems} \label{interpretation}
The previous subsection gives us another point of view of $H^1(C,J_{\mathcal{H}}[2])$ as the set of points over $C$ of a stack. Firstly, for each hyperelliptic curve $H$ over $K(C)$, we associate it with a minimal data $(\mathcal{F}, \underline{c})$, which is determined uniquely up to an action of $\mathbb{G}_m$ on $S$ (see Proposition \ref{minimal data}). Hence, we can view $H$ as an element of $[S/\mathbb{G}_m](C)$, i.e $H$ can be seen as a morphism $\alpha_H: C \rightarrow [S/\mathbb{G}_m].$ The following map:
$$\mathcal{A}= [S/\mathbb{G}_m](C) \rightarrow B\mathbb{G}_m(C),$$whose fiber over a line bundle $\mathcal{F}$ classifies hyperelliptic curves $H$ over $K(C)$ satisfying that $\mathcal{F}_H \cong \mathcal{F}.$ Now we set
$$\mathcal{M}=[BJ_S[2]/\mathbb{G}_m](C),$$and consider the base map
$$b: \mathcal{M}\rightarrow \mathcal{A}.$$By definition, for each $\alpha_H: C \rightarrow [S/\mathbb{G}_m]$, the fiber $\mathcal{M}_{\alpha_H}$ classifies $J_{\mathcal{H}}[2]-$torsors over $C$. In other words, we have 
$$|\mathcal{M}_{\alpha_H}|= | H^1(C, J_{\mathcal{H}}[2])|.$$We also have the following commutative diagram:
\[
\begin{tikzcd}
\mathcal{M} \arrow{dr}{\pi_{\mathcal{M}}} \arrow{rr}{b} && \mathcal{A} \arrow{dl}{\pi_{\mathcal{A}}}\\
 & Hom(C, B\mathbb{G}_m) 
\end{tikzcd}
\]
which implies that for any line bundle $\mathcal{F}$ over $C$,
$$
    |\mathcal{M}_{\mathcal{F}}(k)| = \sum_{H \in \mathcal{A}_{\mathcal{F}}(k)}|H^1(C, J_{\mathcal{H}}[2])| 
     \geq \sum_{H \in \mathcal{A}_{\mathcal{F}}(k)} \frac{|Sel_2(J_H)|}{|H^0(K, J_H[2])|}
$$The above inequality is due to Proposition \ref{selmer group and the first coho}. Moreover, if we denote $\mathcal{A}_{\mathcal{F}}^{\text{sf}}(k)$ to be the set of transversal hyperelliptic curves whose associated minimal line bundles are isomorphic to $\mathcal{F},$ and set $$\mathcal{M}_{\mathcal{F}}^{\text{sf}}(k)=b^{-1}(\mathcal{A}_{\mathcal{F}}^{\text{sf}}(k)),$$then the proposition \ref{selmer group and the first coho} also implies
$$ |\mathcal{M}_{\mathcal{F}}^{\text{sf}}(k)| = \sum_{H \in \mathcal{A}_{\mathcal{F}}^{\text{sf}}(k)}|H^1(C, J_{\mathcal{H}}[2])| 
     = \sum_{H \in \mathcal{A}_{\mathcal{F}}^{\text{sf}}(k)} |Sel_2(J_H)|
$$From the above argument, the main theorems \ref{main theorem 1} and \ref{main theorem 2} will be the corollaries of the following theorems:
\begin{theorem} \label{stack version of main theorem 1}
Assume that $q>4^{1/2(n+1)(2n+1)}$ if $n \geq 7,$ and $q>4^{4.(2n+1)}$ if $n<7.$ If we denote by $\mathcal{A}_{\mathcal{F}}^{\text{min}}(k)$ the subset of $\mathcal{A}_{\mathcal{F}}(k)$ which classifies minimal elements $(\mathcal{F}, \underline{c}),$ and set $\mathcal{M}_{\mathcal{F}}^{\text{min}}(k) = b^{-1}(\mathcal{A}_{\mathcal{F}}^{\text{min}}(k)). $ Then we have the following limits
$$ \limsup_{d \rightarrow \infty} \cfrac{\sum\limits_{\mathcal{F} \in \text{Bun}^{1,\leq d}(C)} |\mathcal{M}^{\text{min}}_{\mathcal{F}}(k)|}{\sum\limits_{\mathcal{F} \in \text{Bun}^{1,\leq d}(C)} |\mathcal{A}^{\text{min}}_{\mathcal{F}}(k)|} \leq 4. \prod_{v \in |C|}(1+|k_v|^{-n-1}) + 2 +f(q),
$$ where $f(q)$ is a rational function of $q$ satisfying that $\lim_{q \rightarrow \infty} f(q) =0.$ If we take the average over the family of semi-stable hyperelliptic curves, then
$$ \limsup_{d \rightarrow \infty} \cfrac{\sum\limits_{\mathcal{F} \in \text{Bun}^{1,\leq d}(C)} |\mathcal{M}^{\text{ss}}_{\mathcal{F}}(k)|}{\sum\limits_{\mathcal{F} \in \text{Bun}^{1,\leq d}(C)} |\mathcal{A}^{\text{ss}}_{\mathcal{F}}(k)|} \leq 6 +f(q),
$$
\end{theorem}
For the transversal case, we may erase the condition of the size of our base field as in the previous theorem, and obtain an exact limit:
\begin{theorem} \label{stack version of main theorem 2}
$$ \lim_{d \rightarrow \infty} \frac{\sum\limits_{\mathcal{F}\in \text{Bun}^{1,\leq d}(C)} |\mathcal{M}_{\mathcal{F}}^{\text{sf}}(k)|}{\sum\limits_{\mathcal{F}\in \text{Bun}^{1,\leq d}(C)} |\mathcal{A}_{\mathcal{F}}^{\text{sf}}(k)|} = 6.
$$
\end{theorem}
It is easy to compute the size of $\mathcal{A}_{\mathcal{F}}(k)$ directly by using Riemann-Roch theorem, but it is not the case for $\mathcal{M}_{\mathcal{F}}(k).$ Fortunately, Corollary \ref{jacobian and orbits} provides us with another description of $\mathcal{M}(k)$: namely, we have an isomorphism
$$\mathcal{M} \cong Hom(C, [V^{\text{reg}}/(G \times \mathbb{G}_m)]).$$Hence, by definition, a $k-$point of $\mathcal{M}$ is an isomorphism class of a triple $(\mathcal{F}, \mathcal{E}, s),$ where $\mathcal{F}$ is a line bundle over $C$, $\mathcal{E}$ is a principal $G-$bundle, and $s$ is a section of $V^{\text{reg}}(\mathcal{E},\mathcal{F})= (V^{\text{reg}}\times^G \mathcal{E}) \otimes \mathcal{F}$. In the next section, by using reduction of $\mathcal{E}$ to some parabolic subgroups of $G$, we will see how to compute the number of sections of $V^{\text{reg}}(\mathcal{E},\mathcal{F})$ precisely.  
\section{Ingredients of the proof of main theorems}
\label{section 3}
In the first subsection, we will compute the portion between the number of regular sections and the number of general sections. Based on Poonen's results in \cite{Poo03} and some arguments in \cite{HLN14}, we deduce the portion by estimating the density of the regular locus. After that, we introduce the canonical reduction of $G-$bundles by associating them with $\text{GSO}(2n+2)-$bundles. As a consequence, we will be able to estimate the size of the automorphism group of a principal $G-$bundle.
\subsection{Density of regular locus}
For each closed point $v \in |C|$, let $K_v$ be the completion of $K$ at $v$, $\mathcal{O}_{K_v}$ be its ring of integers, and $k_v= \mathcal{O}_{K_v}/\varpi_v$ be the residue field at $v$. Then by the result in \cite{HLN14}, if we set $c_v = |V^{\text{reg}}(k_v)|$ then for any $G-$bundle $\mathcal{E}$ we have that 
\begin{equation}
\label{limit of the density of regular locus}
    \lim_{\text{deg}(\mathcal{F}) \rightarrow \infty} \frac{|H^0(\mathcal{E} \times^G V^{\text{reg}} \otimes \mathcal{F})|}{|H^0(\mathcal{E} \times^G V \otimes \mathcal{F})|} = \prod_{v \in |C|} \frac{c_v}{|k_v|^{\text{dim}(V)}}.
\end{equation}
Moreover, we have an upper bound of the above limit:
\begin{proposition} \label{regular locus in general case}
By using the same notation as above, the limit (\ref{limit of the density of regular locus}) is bounded above by
$$\zeta_{C}(n+1)^{-1}.\prod_{i=1}^n \zeta_{C}(2i)^{-1}. \prod_{v \in |C|}(1 + |k_v|^{-n-1}|)$$
\end{proposition}
\begin{proof}
Our ultimate goal is to compute $c_v=|V^{\text{reg}}(k_v)|$ for each $v \in |C|$. Observe that for each invariant $\underline{c} = (c_2,c_3, \dots, c_{2n+2}) \in S(k_v),$ $V_{\underline{c}}^{\text{reg}}(\Bar{k_v})$ has at most two $G(\Bar{k_v})-$orbits. Hence, we can show that $|V_{\underline{c}}^{\text{reg}}(k_v)| \leq 2.|G(k_v)|$ (see Proposition 3.22 in \cite{DVT17}). Additionally, if the polynomial 
$$f_{\underline{c}}(x)=x^{2n+2} + c_2x^{2n}+ \cdots + c_{2n+1}x + c_{2n+2}$$
is in the compliment $k_v[x] \setminus (k_v[x])^2$, the action of $G(\overline{k_v})$ on $V_{\underline{c}}^{\text{reg}}(\overline{k_v})$ is transitive. In fact, if the monic polynomial $f_{\underline{c}}(x)$ is not a square in $k_v[x]$, it will have a root, say $x_1$, in $\overline{k_v}$ of odd order. Followed by [\cite{Wan1}, Proposition 1.41], we have that $PO(2n+2)(\overline{k_v})$ acts on $V_{\underline{c}}^{\text{reg}}(\Bar{k_v})$ transitively. For any $T$ and $T'$ in $V_{\underline{c}}^{\text{reg}}(\Bar{k_v})$, the element $g \in PO(2n+2)$, which satisfies $T = g.T'.g^{-1}$, can be constructed as follows: firstly, we choose an element $h \in \text{GL}(U)$ such that $T'=h.T.h^{-1}$, then 
$$h^*h.T.h^{-1}.(h^*)^{-1} = h^*.T'.(h^*)^{-1} = (h^{-1}.T'.h)^*= T^*=T,$$here we use the fact that $T$ and $T'$ are self-adjoint. Thus, $h^*h$ stabilizes $T$; hence, we can consider it as an element of $L^{\times}= \big(\overline{k_v}[x]/(f_{\underline{c}}(x))\big)^{\times}$. Since we are working on the algebraically closed field $\overline{k_v}$, there exists $t \in L^{\times}$ such that $h^*h = t^2$. Now we can check that the product $h.t^{-1}$ is in $O(2n+2)$, and $T'=ht^{-1}.T.th^{-1}$. Finally, we take $g$ to be the image of $h.t^{-1}$ in $PO(U)$. Now if we set
$$f_{\underline{c}}(x) = \prod_{i=1}^l (x-x_i)^{m_i},$$
then the quotient $L= \overline{k_v}[x]/(f_{\underline{c}}(x))$ is the product of some quotient rings of the form $\overline{k_v}[x]/((x-x_i)^{m_i})$. Without loss of generality, we can assume that $m_1$ is odd. In that case, the element $(-1,1,\dots, 1) \in L$ has determinant $-1$. By multiplying $g$ with that element (if necessary), we can choose an element in $G= \text{PSO}(U)$ that maps $T$ to $T'$ under the conjugation. To sum up, if $f_{\underline{c}}(x)$ is not a square in $k_v[x]$, then $G(\overline{k_v})$ acts on $V_{\underline{c}}(\overline{k_v})$ transitively. As a result, in that case, $|V_{\underline{c}}(k_v)|$ is equal to $|G(k_v)|$ (see the proof of Proposition 3.22 in \cite{DVT17}). We deduce that
\begin{align}
    \label{inequality 1}
    |V^{\text{reg}}(k_v)| \leq \sum_{\underline{c} \in S_1} |G(k_v)| + \sum_{\underline{c} \in S(k_v) \setminus S_1} 2|G(k_v)| 
    = \sum_{\underline{c} \in S(k_v)} |G(k_v)| + \sum_{\underline{c} \in S(k_v)\setminus S_1} |G(k_v)|,
\end{align}
where $S_1 = \{ \underline{c} \in S(k_v) \,\,|\,\,f_{\underline{c}}(x) \in k_v[x]^2 \}$. Now we will finish the proof by computing the size of $S_1$. For any $\underline{c} \in S_1$, we set
$$f_{\underline{c}}(x)= (x^{n+1} + a_2x^{n-1} + \cdots + a_nx + a_{n+1})^2,$$where $a_i \in k_v$. Then $\underline{c}$ is defined uniquely by $n$ coefficients in $k_v$. We deduce that the size of $S_1$ equals to $|k_v|^n$. The upper bound (\ref{inequality 1}) now becomes to
$$c_v=|V^{\text{reg}}(k_v)| \leq |G(k_v)|.(|k_v|^{2n+1} + |k_v|^{n}). $$By taking the product of $c_v$ running over all closed points of $C$, we get an upper bound of the limit (\ref{limit of the density of regular locus}):
\begin{align*}
    \prod_{v \in |C|} \frac{c_v}{|k_v|^{\text{dim}(V)}} &\leq \prod_{v \in |C|} \frac{|G(k_v)|.(|k_v|^{2n+1} + |k_v|^{n})}{|k_v|^{\text{dim}(V)}} \\
    &= \prod_{v \in |C|} (1-\frac{1}{|k_v|^{2}}).(1-\frac{1}{|k_v|^{4}})\dots (1-\frac{1}{|k_v|^{2n}}). (1-\frac{1}{|k_v|^{n+1}}). (1+\frac{1}{|k_v|^{n+1}}) \\
    &= \zeta_{C}(n+1)^{-1}.\prod_{i=1}^n \zeta_{C}(2i)^{-1}. \prod_{v \in |C|}(1 + |k_v|^{-n-1}|). 
\end{align*}
\end{proof}
To compute the average over the set of hyperelliptic curves over $K$, we need to impose the minimality condition. A technique that will be developed in the next sections will allow us to compute the density of regular locus in the minimal case. Precisely, that number is the product of local masses. For example, at a place $v \in |C|,$ the density of tuples $(\mathcal{F}, \underline{c})$ that are minimal at $v$ is (assume that $d = \text{deg}(\mathcal{F})$ is large)
$$1-\frac{|H^0(C, \mathcal{F}(-v)^{\otimes 2} \oplus \mathcal{F}(-v)^{\otimes 3} \oplus \cdots \oplus \mathcal{F}(-v)^{\otimes 2n+2})|}{|H^0(C, \mathcal{F}^{\otimes 2} \oplus \mathcal{F}^{\otimes 3} \oplus \cdots \oplus \mathcal{F}^{\otimes 2n+2})|} = 1 - |k_v|^{-(n+2)(2n+1)}.$$
Hence, we have just proven the following proposition.
\begin{proposition} \label{regular locus in minimal case 1} The density of minimal hyperelliptic surfaces is
	$$\lim_{\text{deg}(\mathcal{F}) \rightarrow \infty} \frac{|H^0(C, S \times^{\mathbb{G}_m} \mathcal{F})^{\text{min}}|}{|H^0(C, S \times^{\mathbb{G}_m} \mathcal{F})|} = \zeta_C((n+2)(2n+1))^{-1}$$
\end{proposition}
Similarly, we also can bound the density of regular sections in the minimal case:
\begin{proposition}
	\label{regular locus in minimal case 2}
	$$ \lim_{\text{deg}(\mathcal{F}) \rightarrow \infty} \frac{|H^0(\mathcal{E} \times^G V^{\text{reg}} \otimes \mathcal{F})^{\text{min}}|}{|H^0(\mathcal{E} \times^G V \otimes \mathcal{F})|} \leq  \zeta_C((n+2)(2n+1))^{-1}.\prod_{i=1}^{n+1} \zeta_{C}(2i)^{-1}.$$
\end{proposition}
\subsection{Density of regular locus in transversal case}
Note that for any $v \in |C|$, if $c \in S(\mathcal{O}_{K_v})$ is transversal, i.e., $\Delta(c) \not\equiv 0 \,\,(\text{mod} \, \varpi_v^2),$ then the reduction $\underline{c} \in S(k_v)$ of $c$ will satisfies that the associated polynomial $f_{\underline{c}}(x)$ is not a square in $k_v[x]$. Consequently, $|V_{\underline{c}}^{\text{reg}}(k_v)|=|G(k_v)|,$ when $c$ is transversal. Now we can estimate the mass of regular locus in the transversal case as follows:
\begin{proposition}
For each $v\in |C|$, we define 
$$\alpha_v= \frac{|\{x \in S(\mathcal{O}_{K_v}/(\varpi_v^2))| \Delta(x) \equiv 0 \hspace{0.2cm}\text{mod} \,(\varpi_v^2) \}|}{|k_v^{4n+2}|}$$and
$$\beta_v= \frac{|\{x \in V(\mathcal{O}_{K_v}/(\varpi_v^2))| \Delta(x) \equiv 0 \hspace{0.2cm}\text{mod} \,(\varpi_v^2) \}|}{|k_v^{2\text{dim}(V)}|},$$
Then we have the following qualities:
\begin{itemize}
\item[i)] $$\lim_{\text{deg}(\mathcal{F})\rightarrow \infty} \frac{|H^0(C,\mathcal{F}^{\otimes 2}\oplus \mathcal{F}^{\otimes 3} \oplus \cdots \oplus \mathcal{F}^{\otimes 2n+2})^{\text{sf}}|}{|H^0(C,\mathcal{F}^{\otimes 2}\oplus \mathcal{F}^{\otimes 3} \oplus \cdots \oplus \mathcal{F}^{\otimes 2n+2})|} = \prod_{v \in |C|}(1-\alpha_v).$$
\item[ii)] $$\lim_{\text{deg}(\mathcal{F})\rightarrow \infty} \frac{|H^0(C,V^{\text{reg}}(\mathcal{E},\mathcal{F}))^{\text{sf}}|}{|H^0(C,V(\mathcal{E},\mathcal{F}))|}= \prod_{v \in |C|}(1-\beta_v)$$
\item[iii)]$$\frac{\prod_{v \in |C|}(1-\beta_v)}{\prod_{v \in |C|}(1-\alpha_v)} = \prod_{v \in |C|}\frac{|G(k_v)|}{|k_v|^{\text{dim}(G)}}$$
\end{itemize}
Here the above script "sf" stands for "square free," i.e., $H^0()^{\text{sf}}$ is the set of sections whose invariants are transversal to the discriminant locus.
\label{regular locus in transversal case}
\end{proposition}
\begin{proof}
	To prove the existence of the limit in $i)$ and $ii),$ we use the result in $[\cite{HLN14}, 5.1.6]$ which is applicable for our case since the polynomial discriminant $\Delta$ is squarefree (see the below lemma \ref{irreducibility of discriminant}). Moreover, the fact that $\Delta$ is geometrically irreducible as a polynomial in $2n+1$ variables which are coordinates of $S$ provides a simple way to know why the limit in $i)$ is non-zero. In fact, if we denote by $\Delta^{smooth}(k_v)$ the smooth locus of the algebraic variety $\Delta(x)=0$ in $S(k_v),$ then \begin{equation}
	|k_v|^{4n+2}\alpha_v= \sum\limits_{\substack{ x \in \Delta^{smooth}(k_v) \\ \Delta(x)= 0}} |k_v|^{2n} + \sum\limits_{\substack{ x \in S(k_v)\setminus \Delta^{smooth}(k_v) \\ \Delta(x)= 0}} |k_v|^{2n+1}.
	\label{alpha}
	\end{equation}
	 Based on the hypothesis that $p>2(2n+2),$ for any fixed values of $a_2, \dots, a_{2n} \in k_v,$ the discriminant polynomial $\Delta(a_{2n+1},a_{2n+2})$ has the following form (see the proof of Lemma \ref{irreducibility of discriminant}):
	$$\Delta(a_{2n+1},a_{2n+2}) = b.a_{2n+1}^{2n+2}+ c.a_{2n+2}^{2n+1}+ g(a_{2n+1},a_{2n+2}),$$ where $b,c \in k_v^*.$ Since $(\Delta, \text{d}_{a_{2n+1}}(\Delta)) =1$ in $k_v(a_{2n+2})[a_{2n+1}],$ there exist $f_1, f_2 \in k_v[a_{2n+1},a_{2n+2}]$ and $f_3\in k_v[a_{2n+2}]^*$ such that
	$$f_1. \Delta + f_2. \text{d}_{a_{2n+1}}(\Delta) = f_3(a_{2n+2}). $$ Moreover, $f_3$ can be chosen such that its degree is O$(n^2).$ Hence, there is at most O$(n^2)$ values of $a_{2n+2}$ satisfying the condition that the set of roots of $$\Delta(a_{2n+1},a_{2n+2}) =  \text{d}_{a_{2n+1}}(\Delta)(a_{2n+1},a_{2n+2}) = 0$$ is non-empty. For each value of $a_{2n+2}$, there is at most $2n+2$ values in $k_v$ of $a_{2n+1}$ such that $\Delta(a_{2n+1},a_{2n+2})=0.$ To sum up, we have just proved that 
	$$|\Delta^{non-smooth}(k_v)|=\text{O}(n^3)|k_v|^{2n-1}.$$
	Substitute it into (\ref{alpha}); we obtain that 
	$$\alpha_v = \text{O}(1/|k_v|^2).$$ Thus, the limit in $i)$ is positive. 
	
	To prove $iii),$ we use the same method as in the proof of Proposition \ref{regular locus in general case} with the following notice: the projection $\pi: V^{\text{reg}} \rightarrow S$ is smooth and surjective; hence, the induced linear map on tangent spaces $\text{d}\pi_x: \text{T}_xV^{\text{reg}} \rightarrow \text{T}_xS$ is surjective. Additionally, for any transversal $s \in S(\mathcal{O}_{K_v}/(\varpi_v^2)),$ we have that $|V^{\text{reg}}_{\bar{s}}(k_v)|=|G(k_v)|,$ where $v\in |C|$ and $\bar{s} \in S(k_v)$ is the reduction of $s$. This helps to compute the local ratios and then deduces $iii)$.
\end{proof}
\begin{lemma} \label{irreducibility of discriminant}(The notation in this lemma is independent to the rest of the paper)
	Let $n>1$ be an integer and $A$ be a unique factorization domain of characteristic 0 or $p$ which is coprime to $n(n-1)$, for each tuple $(a_2, \dots , a_n) \in A^{n-1},$ we denote $\Delta(a_2, \dots, a_n )$ to be the discriminant of the polynomial $f(x) =x^n + a_2x^{n-2} + \cdots + a_n.$ Given any values of $a_2, \dots, a_{n-2} \in A$, the polynomial $\Delta(a_2,\dots,a_{n-1}, a_n) \in A[a_{n-1}, a_n]$ is irreducible in $K[a_{n-1},a_n],$ where $K=\text{Frac}(A)$ In particular, if we consider $\Delta$ as a polynomial with $n-1$ variables then it is irreducible in $K[a_2,\dots, a_n].$
\end{lemma}
\begin{proof}
	The critical point of the proof is that the discriminant polynomial has the following form:$$\Delta(a_2,\dots,a_n) = (-1)^{[\frac{n-1}{2}]}(n-1)^{n-1}a_{n-1}^n+ (-1)^{[\frac{n}{2}]}n^na_n^{n-1}+ g(a_2,\dots,a_n),$$where the degree of $a_{n-1}$ and $a_n$ in $g(a_2,\dots,a_n)$ are less than $n$ and $n-1$ respectively. Based on the condition of the characteristic of $A$, we see that the coefficients of $a_{n-1}^n$ and $a_n^{n-1}$ in $\Delta$ are non-zero. Thus, given fixed values of $a_2, \dots, a_{n-2} \in A$, we may rewrite $\Delta$ in the following simple form:
	$$\Delta_1(a_{n-1},a_n):= \Delta(a_2,\dots,a_n) = b.a_{n-1}^n+ c.a_n^{n-1}+ g(a_{n-1},a_n), $$where $b,c \neq 0,$ $\deg_{a_{n-1}}(g) < n,$ and $\deg_{a_n}(g)<n-1.$ Now we will show that $\Delta_1$ is irreducible in $K[a_{n-1},a_n].$ Firstly, if we consider $a_{n-1}$ and $a_n$ are of weights $n-1$ and $n$ respectively, then it is easy to see that in the above expression of $\Delta$, $b.a_{n-1}^n$ and $c.a_n^{n-1}$ are only highest weight monomials. Moreover, since $(n-1,n)=1,$ the weights of elements in the set $B=\{a_{n-1}^ia_n^j \,\,| \, 0 \leq i \leq n; 0 \leq j \leq n-1\}$ are all different except the case of  $a_{n-1}^n$ and $a_n^{n-1}$. Now assume that we have a proper factorization of $\Delta_1$ in $K[a_{n-1},a_n]$ (in fact, this is in $A[a_{n-1},a_n]$ since $A$ is a UFD):
$$\Delta_1 = g_1(a_{n-1},a_n). g_2(a_{n-1},a_n). $$It is easy to see that every monomial in $g_1$ and $g_2$ is in the set $B$, and $a_{n-1}^n$ and $a_n^{n-1}$ do not appear here. By taking the product of two unique highest weight monomials in $g_1$ and $g_2,$ we get the unique highest weight monomial in their product which is $\Delta_1.$ This is a contradiction since there are two highest monomials in $\Delta_1$.
\end{proof}
By using the above lemma for the case $A=\mathbb{F}_q[t]$ and $[\cite{Poo03}, \text{Theorem 3.4}],$ we proved that
\begin{corollary}
	\label{density of square free discriminant}
	Let $n>1$ be an integer. Then when monic polynomials $f(x) = x^{n} + a_1x^{n-1} + \dots + a_n \in A[x]$ are ordered by $H(f) = max\{|a_1|_t, |a_2|_t^{1/2}, \dots, |a_n|_t^{1/n}\},$ the density of monic polynomials having squarefree discriminant exists and is positive. 
\end{corollary}
\begin{remark}
	\begin{itemize}
		\item[i)] If we assume the abc-conjecture then $[\cite{Poo03}, \text{Theorem 3.1}]$ gives a similar result as in the above corollary for the case $A=\mathbb{Z}.$ Later on, in \cite{bhargavasquarefree}, they proved that result without assuming abc-conjecture. 
		\item[ii)] The density in Corollary \ref{density of square free discriminant} is the product of local densities which can be computed precisely  by using a similar method as in \cite{localdensity}. 
	\end{itemize}
\end{remark}
\subsection{Density of semi-stable curves} \label{semi-stable section}
In this section, we consider the family of hyperelliptic curves over $K(C)$ whose minimal integral models are semi-stable (see [\cite{Liu06}, Definition 10.3.14]). Observe that a minimal integral model $\mathcal{H} \rightarrow C$ is semi-stable if and only if over any closed point $v \in |C|,$ the affine equation of $\mathcal{H}_v: y^2=f(x)$ satisfies that all roots over $\overline{k_v}$ of $f(x)$ are of order at most $2.$ Moreover, if a family $\alpha: C \rightarrow [S/\mathbb{G}_m]$ is semi-stable, it is already minimal. Now we define the semi-stable locus of the affine space $S$ as follows: for each element $\underline{a}=(a_2,\dots, a_{2n+2}) \in S(\overline{k}),$ it is said to be semi-stable if all roots of its associated polynomial $f_{\underline{a}}(x)=x^{2n+2}+a_2x^{2n}+ \dots + a_{2n+2}$ are of order at most $2.$ Similarly, an element $T\in V(\overline{k})$ is called to be semi-stable if its image under the projection map $\pi:V \rightarrow S$ is semi-stable. 
	\begin{proposition}
		\label{constructible}
		The non-semi-stable locus $S^{\text{non-ss}}$ is s constructible subset (i.e., a finite union of locally closed subsets) of $S$ of codimension at least 2. Moreover, for any finite extension $k \subset k_v,$ $S^{\text{non-ss}}(k_v) \subset S(k_v)$ is also of codimension at least 2. The same story for $V.$
	\end{proposition} 
\begin{proof}
	The first statement is a direct consequence of the Chevalley's theorem which states that the image of a constructible set via a finite type morphism of schemes is constructible. In our situation, the non-semi-stable locus $S^{\text{non-ss}}$ is the image of the following polynomial map:
	\begin{align*}
			\phi: \mathbb{A}^{2n-1}(\overline{k}) &\rightarrow  S(\overline{k}) \\
			(c_1,c_2,\dots,c_{2n-1}) & \mapsto (f_2,f_3,\dots,f_{2n+2}),
\end{align*}
where for each $2\leq i \leq 2n+2,$ $f_i$ is the coefficient of $x^{2n+2-i}$ in the following product
$$(x-c_1)^3(x^{2n-1}-c_1x^{2n-2}+c_2x^{2n-3} + \dots + c_{2n-1}).$$ Hence, $S^{\text{non-ss}}$ is constructible. 

By setting $f(x)=x^{2n+2}+a_2x^{2n}+ \dots+ a_{2n+2},$ where $a_i$ are coordinates of $S$, it is easy to see that $S^{\text{non-ss}}$ is the subset of $\Delta(f(x))=\Delta(f'(x))=0.$ Since $\Delta(f)$ and $\Delta(f')$ are relatively prime in $\overline{k}[a_2,\dots,a_{2n+2}],$ the codimension of $S^{\text{non-ss}}$ in $S$ is at least 2. The same argument is applied to $S^{\text{non-ss}}(k_v) \subset S(k_v).$ 
\end{proof}
In this subsection, we consider the family $\mathcal{A}^{\text{ss}}$ of semi-stable integral models $\mathcal{H} \rightarrow C$ such that every fiber is irreducible. It can be shown that a fiber $\mathcal{H}_v: y^2=f(x)$ is irreducible if and only if $f(x)$ is not a square in $k_v[x].$ This is equivalent to that $f(x)$ is not a square in $\overline{k_v}[x]$ because any irreducible factors of $f(x)$ over $k_v[x]$ are separable (since we assume that char$(k) > 2n+2$). We denote by $S^{\square} \subset S$ the subset of elements $\underline{a}$ such that the associated polynomial $f_{\underline{a}}(x)$ is a square in $\overline{k}[x].$ By using a similar argument as in Proposition (\ref{constructible}), we can show that the "square" locus $S^{\square}$ is constructible and of codimension $2$ in $S$. We define $V^{\square}$ similarly and could prove the same statement. Since $V^{\text{non-ss}}\cup V^{\square}$ is the preimage of a constructible set via the $G-$equivariant map $\pi: V \rightarrow S,$ we may express it as a disjoint union of $G-$stable quasi-affine subsets. The quasi-affine property makes sure that the quotient $(V^{\text{non-ss}}\cup V^{\square})\times^G \mathcal{E},$ for any $G-$bundle $\mathcal{E},$ exists as a constructible subset of the vector bundle $V \times^{G} \mathcal{E}.$ The density of $\mathcal{A}^{\text{ss}}$ can be computed as follows:
\begin{proposition} \label{regular locus in semi-stable case}
	Let $\mathcal{F}$ runs over the family of line bundles over $C$. Then for any $G-$bundle $\mathcal{E}$ over $C,$ we have the following limits:
	\begin{itemize}
		\item [i)] \begin{equation}
		\label{density of regular semistable locus}
		\lim_{\text{deg}(\mathcal{F}) \rightarrow \infty} \frac{|\{s \in H^0( V^{\text{reg}}(\mathcal{E},\mathcal{F}) | s \, \text{avoids} \,  (V^{\text{non-ss}}\cup V^{\square})(\mathcal{E}, \mathcal{F})\}|}{|H^0(V(\mathcal{E}, \mathcal{F}))|} = \prod_{v \in |C|} (1-\frac{c_v}{|k_v|^{\text{dim}(V)}}),
		\end{equation}
		where $ (V^{\text{non-ss}}\cup V^{\square})(\mathcal{E}, \mathcal{F}) = ( (V^{\text{non-ss}}\cup V^{\square}) \times^G \mathcal{E}) \otimes \mathcal{F}$ and $c_v= |(V^{\text{non-reg}} \cup V^{\text{non-ss}} \cup V^{\square})(k_v)|.$
		\item [ii)] 
		\begin{equation*}
			\lim_{\text{deg}(\mathcal{F})\rightarrow \infty} \frac{|\{s \in H^0\big(C, (S^{\text{ss}} \times^{\mathbb{G}_m} \mathcal{F})\big) \,|\, s(v) \,\text{is not a square $\forall v \in C$}   \}|}{|H^0(C, S \times^{\mathbb{G}_m} \mathcal{F})|} = \prod_{v \in |C|}(1-\frac{d_v}{|k_v|^{\text{dim}(S)}}),
		\end{equation*}
		where $d_v = |(S^{\text{non-ss}}\cap S^{\square})(k_v)|$ for any $v \in |C|.$
		\item[iii)] $$\frac{\prod_{v \in |C|} (1-\frac{c_v}{|k_v|^{\text{dim}(V)}})}{\prod_{v \in |C|}(1-\frac{d_v}{|k_v|^{\text{dim}(S)}})} = \prod_{v \in |C|}\frac{|G(k_v)|}{|k_v|^{\text{dim}(G)}}$$
	\end{itemize}
\end{proposition}
\begin{proof}
	As in the proof of Proposition \ref{regular locus in general case}, we applied [\cite{HLN14}, Proposition 5.1.1] to calculate the density of regular sections. Precisely, [\cite{HLN14}, Proposition 5.1.1] allows us to compute the density of global sections of a vector bundle that avoid a (locally closed) subscheme of codimension 2. In our current situation, we may prove a similar version of [\cite{HLN14}, Proposition 5.1.1] in which instead of avoiding a locally closed set, our sections avoid a constructible set. In fact, the key point of the proof of [\cite{HLN14}, Proposition 5.1.1] is [\cite{HLN14}, Lemma 5.1.5], and it is easy to generalize [\cite{HLN14}, Lemma 5.1.5] to our case. Hence, $i)$ and $ii)$ are deduced, and note that both limits are non-zero due to the previous proposition.
	
	The result in $iii)$ can be shown similarly as in Proposition \ref{regular locus in general case} by noting that for any $\underline{c} \in S^{\text{ss}}(k_v) \setminus S^{\square}(k_v),$ we have that $V^{\text{reg}}_{\underline{c}}(k_v) = |G(k_v)|.$ 
\end{proof}

\subsection{Canonical reduction of \texorpdfstring{$G-$}{Lg}bundle}
\label{canonical reduction theory}
We have the following isomorphism of algebraic groups over $k$:
$$G=\text{PSO}(U) \cong \text{GSO}(U)/\mathbb{G}_m,$$where $\mathbb{G}_m$ denotes the center of $\text{GSO}(U).$ Immediately, a $G-$bundle could be considered as a $\text{GSO}(U)/\mathbb{G}_m-$bundle. Moreover, since $H^2(C, \mathbb{G}_m)=0$, every $\text{GSO}(U)/\mathbb{G}_m-$bundle over $C$ can be lifted to a $\text{GSO}(U)-$bundle which is well-defined up to tensor twist by a line bundle. To study $\text{GSO}(U)-$bundle, let recall the notation of a line bundle-value bilinear space over $C$.
\begin{definition}
Let $\mathcal{L}$ be a line bundle on $C$, then an $\mathcal{L}-$valued bilinear form on $C$ will mean a triple $(\mathcal{E}, q_{\mathcal{E}}, \mathcal{L}),$ where $\mathcal{E}$ is a vector bundle over $C$ and $q_{\mathcal{E}}: \mathcal{E} \times \mathcal{E} \rightarrow \mathcal{L}$ is an $\mathcal{O}_C-$bilinear morphism. This bilinear form is called symmetric if $q_{\mathcal{E}}: \mathcal{E} \times \mathcal{E} \rightarrow \mathcal{L}$ is symmetric.
 \end{definition}
Now we can define a $\text{GSO}(U)-$bundle as a datum $((\mathcal{E}, q_{\mathcal{E}}, \mathcal{L}), \eta_{\mathcal{E}})$ consisting of an $\mathcal{L}-$valued symmetric bilinear space of rank $n$ (for some line bundle $\mathcal{L}$ on $C$) and an orientation isomorphism $$\eta: \text{disc}(\mathcal{E}, q_{\mathcal{E}}, \mathcal{L}) := \mathcal{E} \otimes (\mathcal{L}^{\vee})^{\otimes n+1} \rightarrow \mathcal{O}_C.$$
By the canonical reduction theory of principal bundles (see \cite{BH04} for definitions and properties), for each $\text{GSO}(U)-$bundle $\mathcal{E}$, there exists uniquely a parabolic subgroup $P \subset \text{GSO}(U)$ with Levi quotient $L$ and the associated $P-$bundle $\mathcal{E}_P$ such that 
\begin{enumerate}
\item We have an isomorphism $\mathcal{E} \cong \mathcal{E}_P(\text{GSO}(U)),$ where $\mathcal{E}_P(\text{GSO}(U))$ is the quotient $(\mathcal{E}_P \times \text{GSO}(U)) / P$ with the following action of $P$ on $\mathcal{E}_P \times \text{GSO}(U):$ for any $h \in P, e \in \mathcal{E}_P,$ and $g \in \text{GSO(U)}$ then $h.(e,g) = (h.e, h^{-1}g).$
\item The Levi bundle $\mathcal{E}_L$ associated, by extension of structure group, with $\mathcal{E}_P$ for the projection $P \rightarrow L$ is semi-stable.

\item For every non-trivial character $\chi$ of $P$ which is a non-negative linear combination of simple roots with respect to some Borel subgroup contained in $P$, the line bundle $\chi_*\mathcal{E}_P$ on $C$ has a positive degree.
\end{enumerate}
Assume that the Levi subgroup $L$ is isomorphic to $$\text{GL}_{n_1} \times \text{GL}_{n_2} \times \cdots \times \text{GL}_{n_t} \times \text{GSO}_{2h},$$ where $n_i >0,$ and $h \geq 0$ are integers satisfying that $2\sum_{i=1}^t n_i + 2h =n.$ In other words, there exists a flag of isotropic subspaces 
$$0 = V_0 \subset V_1 \subset \cdots \subset V_{t} \subset V_{t}^* \subset \cdots \subset V_1^* \subset U,$$
where $\text{dim}(V_i/V_{i-1})=n_i$ for $1\leq i \leq t$, and $\text{dim}(V_{t}^*/V_t)= 2h$. From this, we get a filtration of the vector bundle $\mathcal{E} \times^{\text{GSO}(U)} U$: 
\begin{equation} \label{canonical reduction}0=\mathcal{E}_P \times^P V_0 \subset \mathcal{E}_P \times^P V_1 \subset \cdots \subset \mathcal{E}_P \times^P V_{t} \subset \mathcal{E}_P \times^P V_t^* \subset \cdots \subset \mathcal{E}_P \times^P V_1^*  \end{equation}
satisfying that the quotient bundles $X_i = \mathcal{E}_P \times^P V_i / \mathcal{E}_P \times^P V_{i-1}$ for $1 \leq i \leq t$ and $X_{t+1} = (\mathcal{E}_P \times^P V_t^*) / (\mathcal{E}_P \times^P V_t)$ are semistable. Moreover, $(\mathcal{E}_P \times^P V_{i-1}^*) / (\mathcal{E}_P \times^P V_{i}^*)$ is isomorphic to $X^{\vee}_i \otimes \mathcal{L}$. Remark that $X_{t+1} \cong X_{t+1}^{\vee} \otimes \mathcal{L}$. If we denote the slope of vector bundle $X_i$ by $\mu_i$, then the "canonical conditions" imply that: 
\begin{align}
\mu_1 > \mu_2 > \cdots > \mu_t > \mu_{t+1} = d/2 & \,\,\text{if} \, h>0, \label{slope condition when h>0}\\
\mu_1 > \mu_2 > \cdots > \mu_t \hspace{0.3cm} \text{and} \,\, \mu_{t-1}+\mu_t > d & \,\, \text{if} \, h=0. \label{slope condition when h=0}
\end{align}
We now can estimate the size of the automorphism group of a $\text{GSO}(U)-$bundle as follows: 
\begin{proposition}
There exists a constant $c$ that only depends on the genus $g$ of the curve $C$, such that for any $\text{GSO}(U)-$bundle $(\mathcal{E},q_{\mathcal{E}}, \mathcal{L})$ with the canonical reduction to $P$, we have: 
\begin{itemize}
\item[i)] If $h>0$ then 
\begin{multline} \label{bound of aut group}
|\text{Aut}(E)(\mathbb{F}_q)| \geq  c. |\text{Aut}_{\text{GSO}(2h)}(X_{t+1})|.\prod_{i=1}^t \big(|\text{Aut}_{\text{GL}(n_i)}(X_i)|. |H^0((\wedge^2 X_i)\otimes \mathcal{L}^{\vee})| \big) . \\  \prod_{i=1}^{t-1} \prod_{j=i+1}^{t} \big(|H^0(X_i \otimes X_j \otimes \mathcal{L}^{\vee})| . |H^0(X_i \otimes X_j^{\vee})|\big).\prod_{i=1}^t |H^0(X_i \otimes X_{t+1}^{\vee})|.
\end{multline}
\item[ii)] If $h=0$ and $2\mu_t > d$, where $d = \text{deg}(\mathcal{L}),$ then 
\begin{align*}
|\text{Aut}(E)(\mathbb{F}_q)| \geq  c. \prod_{i=1}^t \big(|\text{Aut}_{\text{GL}(n_i)}(X_i)|. |H^0((\wedge^2 X_i)\otimes \mathcal{L}^{\vee})| \big) . \\  \prod_{i=1}^{t-1} \prod_{j=i+1}^{t} \big(|H^0(X_i \otimes X_j \otimes \mathcal{L}^{\vee})| . |H^0(X_i \otimes X_j^{\vee})|\big).
\end{align*}
\item[iii)] If $h=0$ and $2\mu_t \leq d$, where $d = \text{deg}(\mathcal{L}),$ then 
\begin{align*}
|\text{Aut}(E)(\mathbb{F}_q)| \geq  c. \prod_{i=1}^{t} \big(|\text{Aut}_{\text{GL}(n_i)}(X_i)| \big). \prod_{i=1}^{t-1}\big( |H^0((\wedge^2 X_i)\otimes \mathcal{L}^{\vee})| \big).  \\  |H^0((\wedge^2 X_t^{\vee}) \otimes \mathcal{L})|.\prod_{i=1}^{t-1} \prod_{j=i+1}^{t} \big(|H^0(X_i \otimes X_j \otimes \mathcal{L}^{\vee})| . |H^0(X_i \otimes X_j^{\vee})|\big).
\end{align*}
\end{itemize}
Here for any vector bundle $F$, $H^0(F)$ is the notation of $H^0(C,F)-$the set of global sections over $C$ of $F$.
\end{proposition}
\begin{proof}
The proof is quite similar to the proof of [\cite{DVT17}, Proposition 2.30]. The only thing we should notice is that the statement $iii)$ can be deduced from $ii)$ by interchanging $X_t$ and $X_{t}^{\vee} \otimes \mathcal{L}.$   
\end{proof}
\section{Counting global sections} \label{section 4}
Recall that at the end of Section \ref{interpretation}, we give an interpretation of $k-$points of the moduli space $\mathcal{M}$ via $G-$bundles and their sections. Precisely, for any line bundle $\mathcal{F},$ the fiber $\mathcal{M}_{\mathcal{F}}(k)$ is parametrized by pairs $(\mathcal{E}, s)$, where $\mathcal{E}$ is a $G-$bundle and $s$ is a global section of the associated bundle $(\mathcal{E} \times^G V^{\text{reg}}) \otimes \mathcal{F}.$ In this section, we will estimate the following limit:
$$\limsup_{\text{deg}(\mathcal{F}) \rightarrow \infty} \frac{|\mathcal{M}_{\mathcal{F}}(k)|}{|\mathcal{A}_{\mathcal{F}}(k)|},$$where recall that $\mathcal{A}_{\mathcal{F}}(k)$ classifies the family of Weierstrass curves of Hodge bundle $\mathcal{F}.$ It is easy to see that when $f= \text{deg}(\mathcal{F})$ is large enough, then 
\begin{equation}
\label{number of hyperelliptic curves}
|\mathcal{A}_{\mathcal{F}}(k)|= q^{(2+3+ \cdots + 2n+2)f + (2n+1)(1-g)} = q^{(2(n+1)^2+n)f + (2n+1)(1-g) }.
\end{equation}
Furthermore, any $G-$bundle $\mathcal{E}$ can be considered as a $\text{GSO}(U)/\mathbb{G}_m-$bundle. Hence, $\mathcal{E}$ can be lifted uniquely to a $\text{GSO}(U)-$bundle $(\mathcal{E}', \mathcal{L})$, up to tensor twist by a line bundle, with an associated canonical reduction to a parabolic subgroup $P$ of  $\text{GSO}(U).$ Recall that in Section \ref{canonical reduction theory}, if the Levi quotient of $P$ is isomorphic to $$\text{GL}_{n_1} \times \text{GL}_{n_2} \times \cdots \times \text{GL}_{n_t} \times \text{GSO}_{2h},$$ then there is a filtration of the vector bundle $\mathcal{E}' \times^{\text{GSO}(U)} U$: 
$$ 0=\mathcal{E}_P \times^P V_0 \subset \mathcal{E}_P \times^P V_1 \subset \cdots \subset \mathcal{E}_P \times^P V_{t} \subset \mathcal{E}_P \times^P V_t^* \subset \cdots \subset \mathcal{E}_P \times^P V_1^* $$
satisfying that the quotient bundles $X_i = \mathcal{E}_P \times^P V_i / \mathcal{E}_P \times^P V_{i-1}$ for $1 \leq i \leq t$ and $X_{t+1} = (\mathcal{E}_P \times^P V_t^*) / (\mathcal{E}_P \times^P V_t)$ are semistable. It is also easy to see that our Vinberg-Levi's representation $V$ in \ref{Vinberg representation} is isomorphic to $\text{Sym}^2 \text{std} \otimes \mu^{-1}$ as $\text{GSO}(U)-$representations, where $\mu$ is the multiplier coefficient representation:
\begin{align*}
	\mu: & \hspace{0.3cm}\text{GSO}(U)  \longrightarrow  \mathbb{G}_m \\
	& \hspace{0.8cm} T \hspace{0.7cm} \longmapsto  \mu(T),
\end{align*}where $T.T^* = \mu(T). \text{Id}.$ Thus, we have
$$(\mathcal{E}' \times^{\text{GSO}(U)}V) \otimes \mathcal{F}\cong \text{Sym}_0^2(\mathcal{E}') \otimes \mathcal{L}^{\vee} \otimes \mathcal{F} $$
From the semistable filtration of $\mathcal{E}'$, we find the following "matrix filtration" of $\text{Sym}_{0}^2(\mathcal{E}') \otimes \mathcal{L}^{\vee}$:
 \begin{equation}  \label{V(E)}
 \Large \left( \begin{smallmatrix}
	\text{Sym}^2(X_1) \otimes \mathcal{L}^{\vee} & X_1 \otimes X_2 \otimes \mathcal{L}^{\vee} & \cdots & X_1 \otimes X_t \otimes \mathcal{L}^{\vee} & X_1 \otimes X_{t+1}^{\vee}  & X_1 \otimes X_t^{\vee} & \cdots & X_1 \otimes X_1^\vee \\
X_2 \otimes X_1\otimes \mathcal{L}^{\vee} &	\text{Sym}^2(X_2) \otimes \mathcal{L}^{\vee} &  \cdots &X_2 \otimes X_t \otimes \mathcal{L}^{\vee} & X_2 \otimes X_{t+1}^{\vee} & X_2 \otimes X_t^{\vee} & \cdots & X_2 \otimes X_1^\vee \\
\vdots &	\vdots &   &\vdots & \vdots & \vdots &  & \vdots \\
X_t \otimes X_1 \otimes \mathcal{L}^{\vee} & X_t \otimes X_2\otimes \mathcal{L}^{\vee}	 &  \cdots & \text{Sym}^2(X_t) \otimes \mathcal{L}^{\vee} & X_t \otimes X_{t+1}^{\vee} & X_t \otimes X_t^{\vee} & \cdots & X_t \otimes X_1^\vee \\
X_{t+1}^{\vee} \otimes X_1 &	X_{t+1}^{\vee} \otimes X_2   &  \cdots & X_{t+1}^{\vee} \otimes X_t & \text{Sym}_0^2(X_{t+1}) \otimes \mathcal{L}^{\vee} & X_{t+1} \otimes X_t^{\vee}  & \cdots & X_{t+1} \otimes X_1^\vee  \\
X_t^{\vee} \otimes X_1  & X_t^{\vee} \otimes X_2 &  \cdots & X_t^{\vee} \otimes X_t  & X_t^{\vee} \otimes X_{t+1} & \text{Sym}^2(X_t^{\vee}) \otimes \mathcal{L} & \cdots & X_t^{\vee} \otimes X_1^\vee \otimes \mathcal{L} \\
\vdots &	\vdots &   &\vdots & \vdots & \vdots &  & \vdots \\
X_2^{\vee} \otimes X_1  & X_2^{\vee} \otimes X_2 &  \cdots & X_2^{\vee} \otimes X_t  & X_2^{\vee} \otimes X_{t+1}  & X_2^{\vee} \otimes X_t^{\vee} \otimes \mathcal{L} & \cdots & X_2^{\vee} \otimes X_1^\vee \otimes \mathcal{L} \\
X_1^{\vee} \otimes X_1  & X_1^{\vee} \otimes X_2 &  \cdots & X_1^{\vee} \otimes X_t  & X_1^{\vee} \otimes X_{t+1}  & X_1^{\vee} \otimes X_t^{\vee} \otimes \mathcal{L} & \cdots & \text{Sym}^2(X_1^{\vee}) \otimes \mathcal{L}
\end{smallmatrix} \right) 
\end{equation}

Thus, if we denote the above matrix by $A = (a_{ij})_{1 \leq i,j \leq 2t+1},$ then we deduce the following bound of $|\mathcal{M}_{\mathcal{F}, \mathcal{E}}(k)|:$ 
\begin{align}
\label{bound of M}
|\mathcal{M}_{\mathcal{F}, \mathcal{E}}(k)| & = |H^0((\mathcal{E} \times^G V^{\text{reg}}) \otimes \mathcal{F})| \leq |H^0((\mathcal{E}' \times^{\text{GSO}(U)} V) \otimes \mathcal{F})| \nonumber \\
& \leq |H^0(\text{Sym}_0^2 \big(X_1 \oplus \cdots \oplus X_t \oplus X_{t+1} \oplus (X_t^{\vee}\otimes \mathcal{L}) \oplus \cdots \oplus (X_1^{\vee}\otimes \mathcal{L})\big) \otimes \mathcal{L}^{\vee} \otimes \mathcal{F} )| \nonumber \\
& = \prod_{1 \leq i \leq j \leq 2t+1} |H^0(a_{ij} \otimes \mathcal{F})|. 
\end{align}  
\begin{remark}
	In case $h=0$, i.e., we do not have the bundle $X_{t+1},$ the $2t \times 2t$ matrix $B$, that is obtained by removing $X_{t+1}-$related entries from $A,$ is the semistable filtration of $\text{Sym}^2(\mathcal{E}') \otimes \mathcal{L}^{\vee}$. Hence, to bound the size of $\mathcal{M}_{\mathcal{F}, \mathcal{E}}(k)$, we need to consider traceless global sections of $B$. 
\end{remark}
Given a parabolic subgroup $P \subset \text{GSO}(U),$ we denote by $\mathcal{M}_{\mathcal{F},P}(k) = \bigcup\limits_{\mathcal{E} \in \text{Bun}^P(k)}\mathcal{M}_{\mathcal{F},\mathcal{E}}(k),$ where $\text{Bun}^P(k)$ is the set of $\text{GSO}(U)-$bundles having canonical $P-$reductions. Keeping the notations as in Section \ref{canonical reduction theory} and set $\text{deg}(\mathcal{F})=f$, Table \ref{table 1} summarizes the contribution of $\mathscr{M}_{\mathcal{F},P}(k)$ to the average:
\begin{table}[ht]
	\centering
	\begin{tabular}{|c|c|c|c|c|}	
		\hline
		Case &  Hypothesis & Subcase & Subsubcase  & Contr. \\
		\hline 
		\multirow{2}{*}{1}& \multirow{2}{3.4cm}{\hspace{1cm}$h \neq 0,$ \\ $-2\mu_1 +d +f \leq 2g-2$   } & &  &$0$ \\
		&&&& \\
		\hline
		\multirow{4}{*}{2} & \multirow{4}{3.4cm}{ \\$\hspace{0cm}h=0; \mu_t \geq d/2;$ \\ $-2\mu_1 +d +f \leq 2g-2$ } & $2\mu_1 - d \leq (4t-3)f/2 $& & 0  \\ \cline{3-5}  
		
		&& \multirow{3}{*}{$2\mu_1 - d > (4t-3)f/2 $}&$P \neq B$& 0 \\ \cline{4-5}
		&&&$P=B$; $2\mu_1 -d = (2n+1)f$& 1\\ \cline{4-5}
		&&&$P=B$; $2\mu_1 -d < (2n+1)f$ & $f_1(q)$\\
		\hline
		\multirow{4}{*}{3} & \multirow{4}{3.4cm}{ \\$h=0; \mu_t<d/2$ \\ $-2\mu_1+d+f \leq 2g-2$ } & $2\mu_1 - d \leq (4t-3)f/2 $& & 0  \\ \cline{3-5}  
		
		&& \multirow{3}{*}{$2\mu_1 - d > (4t-3)f/2 $}&$P \neq B$& 0 \\ \cline{4-5}
		&&&$P=B$; $2\mu_1 -d = (2n+1)f$& 1\\ \cline{4-5}
		&&&$P=B$; $2\mu_1 -d < (2n+1)f$ & $f_2(q)$\\
		\hline
		4 & \multirow{1}{3.4cm}{$-2\mu_1+d+f >2g-2$ }&& &4 \\
		\hline
	\end{tabular}
	\vskip1ex
	\caption{Contribution to the average.} 	\label{table 1}
\end{table}
\subsection{The case \texorpdfstring{$h \neq 0$}{Lg} and \texorpdfstring{$-2\mu_1 +d + f \leq 2g-2$}{Lg}}
The computation in this subsection is quite similar to the case of odd hyperelliptic curves (see section 2.7 of \cite{DVT17}). In fact, based on the canonical condition of slopes:
$$\mu_1 > \mu_2 > \cdots > \mu_t > \mu_{t+1} =d/2,$$
and the inequality (\ref{bound of M}), we can prove the following lemma:
\begin{lemma}
\label{lemma of zero contribution}
Let $P_1$ be the set of $G-$bundles whose associated canonical reductions (with the notations as above) satisfy that $0 < \mu_i - \mu_{i+1} \leq f$ for all $1 \leq i \leq t$, and $f + \mu_{e+1} < \mu_{1} \leq f + \mu_e$ for some index $ 2 \leq e \leq t,$ where $f = \text{deg}(\mathcal{F}).$ Then 
$$\lim_{\text{deg}(\mathcal{F}) \rightarrow \infty} \sum_{\mathcal{E} \in P_1} \frac{|\mathcal{M}_{\mathcal{F},\mathcal{E}}(k)|}{|\mathcal{A}_{\mathcal{F}}(k)|} = 0$$
\end{lemma}
\begin{proof}
We will prove it by induction. With the hypothesis as in this lemma, by looking at the $X_1-$related part $g(X_1)$ of the right-hand side of the inequalities (\ref{bound of aut group}) and (\ref{bound of M}), we can bound above the $X_1-$related part of 
$$\frac{|\mathcal{M}_{\mathcal{F},\mathcal{E}}(k)|}{|\text{Aut}(\mathcal{E})(k)||\mathcal{A}_{\mathcal{F}}(k)|}$$by
\begin{equation}
     \frac{h_1(\mu_1,f)}{h_2(\mu_1,f)},
\end{equation}
where
\begin{multline*}
    h_1(\mu_1,f) = |H^0(\text{Sym}^2(X_1) \otimes \mathcal{L}^{\vee} \otimes \mathcal{F})|. \prod_{i=2}^{t}\big(|H^0(X_1 \otimes X_i \otimes \mathcal{L}^{\vee} \otimes \mathcal{F}| \big).   |H^0(X_1 \otimes X_{t+1}^{\vee} \otimes \mathcal{F})|\\ . \prod_{i=2}^{t}\big(|H^0(X_1 \otimes X_i^{\vee} \otimes \mathcal{F}| \big). \prod_{i=1}^{e}\big(|H^0(X_1^{\vee} \otimes X_i \otimes \mathcal{F}| \big),
\end{multline*}
and 
\begin{multline*}
    h_2(\mu_1,f) = |H^0((\wedge^2X_1)\otimes \mathcal{L}^{\vee})|. \prod_{i=2}^t\big(|H^0(X_1\otimes X_i \otimes \mathcal{L}^{\vee})|.|H^0(X_1 \otimes X_i^{\vee})| \big). |H^0(X_1 \otimes X_{t+1}^{\vee})| \\ .q^{(4n_1(2n+2-n_1)+2n_1)f}.
\end{multline*}
By Riemann-Roch theorem for semi-stable vector bundles (see [\cite{Tho16}, Lemma 4.4]), when $\text{deg}(\mathcal{F})$ is large enough, we could estimate $h_1$ and $h_2$ as follows:
\begin{align*}
log_q(h_1(\mu_1,f)) &\approx \frac{n_1(n_1+1)}{2}(2\mu_1+f-d)+ \sum_{i=2}^t \big(n_1n_i(\mu_1+\mu_i+f-d)\big) + 2hn_1.(\mu_1-d/2\\
& \hspace{2cm}+f)+\sum_{i=2}^t \big(n_1n_i(\mu_1-\mu_i+f)\big) + \sum_{i=1}^e \big(n_1n_i(\mu_i -\mu_1 + f) \big), \\
log_q(h_2(\mu_1,f))&\approx \frac{n_1(n_1-1)}{2}(2\mu_1-d)+ \sum_{i=2}^t \big(n_1n_i(\mu_1+\mu_i-d)+n_1n_i(\mu_1-\mu_i)\big) + 2hn_1.(\mu_1 \\
 &  \hspace{6cm}-d/2) +\big(4n_1(2n+2-n_1)+2n_1\big)f.
\end{align*}
Hence,
\begin{align*}
log_q\Big(\frac{h_1(\mu_1,f)}{h_2(\mu_1,f)} \Big) & \approx n_1(2\mu_1 -d)+ \sum_{i=2}^e\big( n_1n_i(\mu_i-\mu_1)\big) - f\big(n_1n_{e+1}+\dots + n_1n_t + \\ & \hspace{3cm}+ 2hn_1+ n_1n_t + \cdots + n_1n_2 + \binom{n_1+1}{2}\big) \\
&\leq n_1\big(2\mu_1-d + (\mu_e-\mu_1)-f(2t+2-e) \big) \,\,\,\, \big(\text{since} \,\,  e\geq 2, n_i \geq 1 \big)\\
& \leq n_1 \big( 2(t-e+1)f + f - f(2t+2-e) \big)  \\
&\leq fn_1(1-e) \,\,\, \big(\text{since} \,\, \mu_1 < \mu_e + f, \,\text{and} \,\, \mu_e \leq f(t-e+1)\big).
\end{align*}
This proves the proposition.
\end{proof}
Now we are going to consider the case $\mu_1 \leq f + \mu_{t+1} = f + d/2$. In this case, there exists an index $1 \leq e \leq t$ such that $\mu_1 + \mu_e -d > f \geq \mu_1 + \mu_{e+1} -d.$ By using a similar argument as in the proof of Proposition \ref{lemma of zero contribution}, we can prove that the contribution of this case to the average equals to zero since the $X_1-$related part of our limit is:
$$\frac{h_1(\mu_1,f)}{h_2(\mu_1,f)},$$
where 
\begin{align*}
log_q(h_1(\mu_1,f))  \approx \frac{n_1(n_1+1)}{2}(2\mu_1+f-d)+ \sum_{i=2}^t \big(n_1n_i(\mu_1+\mu_i+f-d)\big) + 2hn_1.(\mu_1-d/2+f)  \\
+\sum_{i=2}^t \big(n_1n_i(\mu_1-\mu_i+f)\big)+ \sum_{i=1}^{t+1} \big(n_1n_i(\mu_i -\mu_1 + f) \big)+\sum_{i=e+1}^{t} \big( n_in_1(-\mu_i  -\mu_1 +d +f) \big) ,   \\
log_q(h_2(\mu_1,f)) \approx \frac{n_1(n_1-1)}{2}(2\mu_1-d)+ \sum_{i=2}^t \big(n_1n_i(\mu_1+\mu_i-d)+n_1n_i(\mu_1-\mu_i)\big) \hspace{2.4cm}\\+ 2hn_1.(\mu_1-d/2) 
+(4n_1(2n+2-n_1)+2n_1)f. \hspace{2.4cm}
\end{align*}
Hence,
\begin{align*}
log_q\Big(\frac{h_1(\mu_1,f)}{h_2(\mu_1,f)} \Big) \approx n_1(2\mu_1 -d)+ \sum_{i=2}^{t+1}\big( n_1n_i(\mu_i-\mu_1)\big) + \sum_{i=e+1}^{t} \big( n_in_1(-\mu_i  -\mu_1 +d) \big) \hspace{2cm} \\- f\big(n_1n_{e}+\dots + n_1n_2 + \binom{n_1+1}{2}\big) \hspace{2cm}\\
\leq n_1\big(2\mu_1-d + n_{t+1}(\mu_{t+1}-\mu_1)+\sum_{i=e+1}^{t} \big( n_i(-\mu_i  -\mu_1 +d) \big) -f(n_e + \dots + n_2 + \frac{n_1+1}{2}) \big) \hspace{1.2cm}\\
\leq \frac{n_1}{2} \big( 2(2\mu_1-d) -2h(2\mu_1-d) + \sum_{i=e+1}^{t} \big( n_i(-2\mu_i +d -2\mu_1 +d) \big) -4f(n_e + \dots + n_2 + \frac{n_1+1}{2}) \big)  \\
\leq -fn_1e \,\,\,\,(\text{since} \,\, \mu_i > d/2 \,\,\forall i, \,\text{and} \,\,\,h \neq 0). \hspace{9cm}
\end{align*}
This implies that the case $f/2 + d/2 < \mu_1  \leq f +d/2$ has no contribution in the average. 

If $0 \leq -2\mu_1 +d +f \leq 2g-2$, then by using a similar computation as above, we also obtain the zero contribution to the average in this case. 
\subsection{The case \texorpdfstring{$h=0$}{Lg}, \texorpdfstring{$-2\mu_1 +d +f \leq 2g-2$}{Lg}, and \texorpdfstring{$\mu_t \geq d/2$}{Lg} } \label{case 2}
Similar to the proposition \ref{lemma of zero contribution} in the previous case, we can prove the following lemma:
\begin{lemma}
\label{zero contribution in case h=0}
Assume that $\mathcal{E}$ is a $G-$bundle associated with a canonical reduction that was described as in the previous subsection. Suppose that $\mu_t \geq d/2$, and $0 < \mu_i - \mu_{i +1} \leq f$ for all $1 \leq i \leq t-1$. If we denote by $P_2$ the set of $\mathcal{E}$ such that $P \neq B$, or $P =B$ and $\mu_1 \leq f + \mu_3$, then 
$$\lim_{\text{deg}(\mathcal{F}) \rightarrow \infty} \sum_{\mathcal{E} \in P_2} \frac{|\mathcal{M}_{\mathcal{F},\mathcal{E}}(k)|}{|\mathcal{A}_{\mathcal{F}}(k)|} = 0$$
\end{lemma}
\begin{proof}
This can be done by a similar computation as in the proof of Proposition \ref{lemma of zero contribution}.
\end{proof}
From the above result and the proof of Proposition \ref{lemma of zero contribution}, we can see that the contribution to the average is positive only if the parabolic canonical reduction of our bundle $\mathcal{E}$ is associated with the Borel subgroup, i.e., $n_i =1$ for all $1 \leq i \leq n+1$, and $f \geq \mu_i - \mu_{i+1} \geq f/2$ for all $1 \leq i \leq n$. Moreover, in case $\mu_{n+1} > d/2$, we need to have that $2\mu_{n+1} > d + f/2$. Hence, the semistable filtration (\ref{canonical reduction}) of $\mathcal{E}$ splits if $f$ is large enough. Additionally, from the smoothness of generic fibers, if $\mu_i - \mu_{i +1} =f $ for any $1 \leq i \leq n$ then $X_i \cong  X_{i+1} \otimes \mathcal{F}.$ By using (\ref{bound of aut group}) and (\ref{bound of M}), we find that:
\begin{multline*}
|\mathcal{M}_{\mathcal{F},\mathcal{E}}(k)| = h_1(g,X_2, \dots, X_{n+1}). |H^0(X_2 \otimes X_1^{\vee} \otimes \mathcal{F})| q^{\sum_{i=1}^{n+1}\big( (\mu_1 + \mu_i -d + f) + (\mu_1 -\mu_i + f)\big ) + (2n+2)(1-g)} \\ \hspace{12cm}.|\text{Aut}(\mathcal{E})(k)|^{-1} \\
= h_1(g,X_2, \dots, X_{n+1}). |H^0(X_2 \otimes X_1^{\vee} \otimes \mathcal{F})|. |\text{Aut}(\mathcal{E})(k)|^{-1}. q^{(2n+2)\mu_1 - (n+1)d + (2n+2)f +(2n+2)(1-g)},
\end{multline*}
\begin{align*}
|\text{Aut}(\mathcal{E})(k)|& = h_3(g,X_2, \dots, X_{n+1}). |\text{Aut}(X_1)|. \prod_{i=2}^{n+1}\big( |H^0(X_1 \otimes X_i \otimes \mathcal{L}^{\vee})|.|H^0(X_1 \otimes X_i^{\vee})| \big) \\
& = h_4(g,X_2, \dots, X_{n+1}). (q-1) q^{n(2\mu_1 -d) + 2n(1-g)},
\end{align*}
and
$$ |\mathcal{A}_{\mathcal{F}}(k)| = h_2(g,X_2, \dots, X_{n+1}). q^{(4n+3)f + 2(1-g)},$$
where $h_i$ is a function that depends only on $X_2, \dots X_{n+1},$ and $g-$the genus of $C$. Hence, the contribution of this case to the average is (we consider the $X_1-$related part firstly, then the rest can be done inductively):
\begin{align*}
    A & = \sum_{\mu_2 + f/2 < \mu_1 < \mu_2 + f} \hspace{0.3cm} \sum\limits_{X_1 \in \text{Bun}_{\mu_1,1}(C)(k)} \frac{|H^0(X_2 \otimes X_1^{\vee} \otimes \mathcal{F})|. q^{(2n+2)\mu_1 - (n+1)d + (2n+2)f +(2n+2)(1-g)}}{(q-1)q^{(4n+3)f+2(1-g)}.q^{n(2\mu_1-d) +2n(1-g)}}  \\
    & \hspace{2cm}+ \frac{|H^0(\mathcal{O}_C)^*|. q^{(2n+2)(\mu_2+f) - (n+1)d + (2n+2)f +(2n+2)(1-g)}}{(q-1)q^{(4n+3)f+2(1-g)}.q^{n(2\mu_2 +2f-d) +2n(1-g)}}\\
    & =\sum_{\mu_2 + f/2 < \mu_1 < \mu_2 + f}  \frac{|\text{Sym}_C^{\mu_2-\mu_1+f}(\mathbb{F}_q)|. q^{2\mu_1-d}}{(q-1)q^{(2n+1)f}}\hspace{0.5cm}  + \hspace{0.5cm} \frac{q^{2\mu_2-d}}{q^{(2n-1)f}} \\
    & \leq  T. \sum_{i=[\frac{f}{2}]}^{f-1} \frac{q^{2\mu_2-d+i}}{(q-1)q^{2nf}} +\frac{q^{2\mu_2-d}}{q^{(2n-1)f}} , 
\end{align*}
where $T$ is a constant that depends only on $g$ and $n$. From the condition of slopes $\mu_i$ ($1\leq i \leq n+1$), we deduce that $2\mu_2-d \leq (2n-1)f$, and the equality holds only if $\mu_i=\mu_{i+1} +f$ for all $2\leq i \leq n$ and $2\mu_{n+1} -d = f$. Thus, $A$ will be bounded by $\frac{T}{(q-1)^2} + 1/q$  if $2\mu_2 -d <(2n-1)f$. Otherwise, if $2\mu_2-d = (2n-1)f$ then $A$ will be bounded by $1+ \frac{T}{(q-1)^2}$. By induction, we conclude that the contribution of this case to the average is bounded by $1 + f(q),$ where $lim_{q \rightarrow \infty} f(q) = 0.$ Moreover, the case that contributes 1 can be seen by the following proposition:
\begin{proposition}Assume that the $G-$bundle $\mathcal{E}$ has the canonical reduction to the Borel subgroup. Let $X_i$ be the associated line bundle as in the filtration (\ref{canonical reduction}) of $\mathcal{E}$ for $1 \leq i \leq n+1$, and denote $\mu_i$ the degree of $X_i$. If $\mu_i=\mu_{i+1} +f$ for all $1\leq i \leq n$ and $2\mu_{n+1} -d = f$, then any sections of $(\mathcal{E} \times^G V^{\text{reg}})\otimes \mathcal{F}$, whose generic fibers have nonzero discriminant, will factor through the Kostant section $\kappa_1.$ As a consequence, the contribution of this case to the average equals to 1.  
	\label{factor through kostant 1}
\end{proposition}
\begin{proof}
First, we notice that for any $1 \leq i \leq n$, $X_i \otimes X_{i+1}^{\vee} \otimes \mathcal{F}$ is a line bundle of degree 0, and so is $X_{n+1}^{\otimes -2} \otimes \mathcal{L} \otimes \mathcal{F}$. Hence, to satisfy that the generic fiber has nonzero discriminant, all the above line bundles need to be trivial. In this case, any sections are of the following form:
$$\begin{pmatrix}
* & * & \cdots & * & * & *  \\
x_1 & * & \cdots & * & * & *  \\
0 & x_2 & \cdots & * & * & * \\
\vdots & \vdots & \ddots & \vdots & \vdots & \vdots  \\
0 & 0 & \cdots & x_2 & * & *  \\
0 & 0 & \cdots & 0 & x_1 & * 
\end{pmatrix}
$$where $x_i \in H^0(C, \mathcal{O}_C)^* = k^*.$ Let we denote $A$ the above matrix, then firstly we will try to transform all entries in the sub-diagonal to $1$. That can be done as follows: 
\footnotesize
\[ diag(a_1,\dots, a_n,1,a_n^{-1}, \dots,a_1^{-1}).
\left( \begin{array}{cccccc}
* & * & \cdots & * & * & *  \\
x_1 & * & \cdots & * & * & *  \\
0 & x_2 & \cdots & * & * & * \\
\vdots & \vdots & \ddots & \vdots & \vdots & \vdots  \\
0 & 0 & \cdots & x_2 & * & *  \\
0 & 0 & \cdots & 0 & x_1 & * 
 \end{array} \right).diag(a_1^{-1},\dots, a_n^{-1},1,a_n,\dots,a_1)\]
 \[=\left( \begin{array}{cccccc}
* & * & \cdots & * & * & *  \\
a_1^{-1}a_2x_1 & * & \cdots & * & * & *  \\
0 & a_2^{-1}a_3x_2 & \cdots & * & * & * \\
\vdots & \vdots & \ddots & \vdots & \vdots & \vdots  \\
0 & 0 & \cdots & a_2^{-1}a_3x_2 & * & *  \\
0 & 0 & \cdots & 0 & a_1^{-1}a_2x_1 & * 
 \end{array} \right)\]  \normalsize
 Now we can finish this step by taking $a_n :=x_n ; a_{n-1}:=x_nx_{n-1}; \dots; a_1 :=x_nx_{n-1}\dots x_1.$ \\
 In the next step, we will make the first row and the last column to become the desired form, i.e., all entries in the first row are zero except two last entries. To do that we consider the following matrix in $\text{SO}(U)$:
 \[ C=\left( \begin{array}{cccccc}
1 & a_1 & a_2 & \cdots & a_{2n-1} & b  \\
0 & 1 & 0 & \cdots & 0 & -a_{2n-1}   \\
0 & 0 & 1 & \cdots & 0 & -a_{2n-2} \\
\vdots & \vdots & \vdots & \ddots &  \vdots & \vdots  \\
0 & 0 & 0 & \cdots  & 1 & -a_1  \\
0 & 0 & 0 & \cdots  & 0 & 1 
 \end{array} \right) \]
 where $b=-\dfrac{1}{2}\sum_{i=1}^{2n-1}a_ia_{2n-i}$. Then we can choose $a_i$ such that 
\[ C . \left( \begin{array}{cccccc}
* & * & \cdots & * & * & *  \\
1 & * & \cdots & * & * & *  \\
0 & 1 & \cdots & * & * & * \\
\vdots & \vdots & \ddots & \vdots & \vdots & \vdots  \\
0 & 0 & \cdots & 1 & * & *  \\
0 & 0 & \cdots & 0 & 1 & * 
 \end{array} \right) . C^{-1} = \left( \begin{array}{cccccc}
0 & 0 & \cdots & 0 & * & *  \\
1 & * & \cdots & * & * & *  \\
0 & 1 & \cdots & * & * & 0 \\
\vdots & \vdots & \ddots & \vdots & \vdots & \vdots  \\
0 & 0 & \cdots & 1 & * & 0  \\
0 & 0 & \cdots & 0 & 1 & 0 
 \end{array} \right) = E. \]
 
Similarly, by considering a matrix of the form
\[ D=\left( \begin{array}{cccccccc}
1 & 0 & 0 & 0 & \cdots & 0 & 0 & 0 \\
0 & 1 & a_1 & a_2 & \cdots & a_{2n-3} & b & 0 \\
0 & 0 & 1 & 0 & \cdots & 0 & -a_{2n-3} & 0  \\
0 & 0 & 0 & 1 & \cdots & 0 & -a_{2n-4} & 0 \\
\vdots & \vdots & \vdots & \vdots & \ddots &  \vdots & \vdots & \vdots \\
0 & 0 & 0 & 0 & \cdots  & 1 & -a_1 & 0  \\
0 & 0 & 0 & 0 & \cdots  & 0 & 1 & 0 \\
0 & 0 & 0 & 0 & \cdots  & 0 & 0 & 1
 \end{array} \right) \]
 we can choose $a_i$ such that $D.E.D^{-1}$ has the desirable second row. By doing the same way, after $n$ steps we will obtain a matrix that belongs to our Kostant section $\kappa_1$. 
\end{proof}
\subsection{The case \texorpdfstring{$h=0$}{Lg}, \texorpdfstring{$-2\mu_1 +d +f \leq 2g-2$}{Lg}, and \texorpdfstring{$\mu_t < d/2$}{Lg}}
This case can be done as the same as the previous case with an only small change that in the whole argument and computation, we replace $X_t$ by $X^{\vee} \otimes \mathcal{L}$. The contribution of this case to the average is $1 + O(q^2),$ where the constant 1 is for the case that $n_i =1$ for all $1 \leq i \leq t=n+1,$ $\mu_i = \mu_{i+1} +f $ for all $1 \leq i \leq n$ and $d - 2\mu_{n+1} = f.$ In that case, any regular sections of $(\mathcal{E} \times^G V^{\text{reg}}) \otimes \mathcal{F}$ factor through the Kostant section $\kappa_2$ (c.f. Proposition \ref{factor through kostant 1}).

\subsection{The case \texorpdfstring{$-2\mu_1 +d +f > 2g-2$}{Lg} or \texorpdfstring{$\mathcal{E}$}{Lg} is semi-stable} \label{case 4} We can see that $H^0((V \times^G \mathcal{E}) \otimes \mathcal{F}) = H^0(\text{Sym}^2_0(\mathcal{E} \otimes \mathcal{L}^\vee \otimes \mathcal{F})$ equals to the direct sum of the sets of global sections of vector bundles in the "matrix filtration" of $\text{Sym}^2_0(\mathcal{E})\otimes \mathcal{L}^\vee \otimes \mathcal{F}$ (see the matrix before (\ref{bound of M})) since all vector bundles there have degrees bigger than $2g-2$. More precisely,
\begin{align*}
    \text{dim}(H^0(\text{Sym}^2_0(\mathcal{E}) \otimes \mathcal{L}^\vee \otimes \mathcal{F})) &= \text{deg}(\text{Sym}^2_0(\mathcal{E}) \otimes \mathcal{L}^\vee \otimes \mathcal{F}) + (1-g).\text{rank}(\text{Sym}^2_0(\mathcal{E})) \\
&= f . \text{rank}(\text{Sym}^2_0(\mathcal{E})) + (1-g).\text{rank}(\text{Sym}^2_0(\mathcal{E}))  \\
   & = ((2n+3)(n+1)-1)(f+1-g).
\end{align*}
In the space of $G-$bundles $\text{Bun}_G(k)$, we denote $d\mathcal{E}$ to be the counting measure, weighted by the sizes of the automorphism groups. Then we have the following qualities:
\begin{align}
&\lim_{f \rightarrow \infty} \dfrac{\mathlarger{\int}_{\text{Bun}_G^{\mu_1<f/2+d/2-g+1}(k)}|H^0(\text{Sym}^2_0(\mathcal{E}) \otimes \mathcal{L}^\vee \otimes \mathcal{F})|\,\,d\mathcal{E}}{\mathcal{A}_{\mathcal{F}}(k)} \nonumber \\
=\hspace{0.5cm}& \lim_{f \rightarrow \infty}\frac{ \mathlarger{\int}_{\text{Bun}_G^{\mu_1<f/2+d/2-g+1}(k)} q^{((2n+3)(n+1)-1)(f+1-g)}\,\, d\mathcal{E}}{q^{(2(n+1)^2+n)f + (2n+1)(1-g) }} \nonumber\\
=\hspace{0.5cm}& \lim_{f \rightarrow \infty}\frac{|\text{Bun}_G^{\mu_1<f/2+d/2-g+1}(k)|. q^{((2n+3)(n+1)-1)(f+1-g)}\,\, }{q^{(2(n+1)^2+n)f + (2n+1)(1-g) }} \nonumber \\
=\hspace{0.5cm}&  |\text{Bun}_G(k)|.q^{(1-g)(2n^2+3n+1)} \nonumber \\
=\hspace{0.5cm}& 4.q^{\text{dim}(G).(g-1)}. \zeta_C(n+1). \prod_{i=1}^n(\zeta_C(2i)). q^{(1-g)(2n^2+3n+1)} \nonumber \\
&(\text{since the Tamagawa number of $G=\text{PSO}(2n+2)$ is $4$.} ) \nonumber \\
=\hspace{0.5cm}& 4.\zeta_C(n+1). \prod_{i=1}^n(\zeta_C(2i)) \label{equality 10}
\end{align}
By using the dominated convergence theorem, we have that
\begin{align*}
	&\lim_{f \rightarrow \infty} \sum_{\mathcal{E} \in \text{Bun}_G^{\mu_1<f/2+d/2-g+1}(k)} \frac{|\mathcal{M}_{\mathcal{E}, \mathcal{F}}(k)|}{|\mathcal{A}_{\mathcal{F}}(k)|} = \lim_{f \rightarrow \infty} \frac{\mathlarger{\int}\limits_{\text{Bun}_G^{\mu_1<f+d/2-g+1}(k)} |H^0(\text{Sym}^2_0(\mathcal{E}) \otimes \mathcal{L}^\vee \otimes \mathcal{F})^{\text{reg}}|d\mathcal{E}}{|\mathcal{A}_{\mathcal{F}}(k)|} \\
	&= \lim_{f \rightarrow \infty} \frac{\mathlarger{\int}\limits_{\text{Bun}_G^{\mu_1<f+d/2-g+1}(k)} \mathlarger{\frac{|H^0(\text{Sym}^2_0(\mathcal{E}) \otimes \mathcal{L}^\vee \otimes \mathcal{F})^{\text{reg}}|}{|H^0(\text{Sym}^2_0(\mathcal{E}) \otimes \mathcal{L}^\vee \otimes \mathcal{F})|}}.|H^0(\text{Sym}^2_0(\mathcal{E}) \otimes \mathcal{L}^\vee \otimes \mathcal{F})| d\mathcal{E}}{|\mathcal{A}_{\mathcal{F}}(k)|} 
	\end{align*}
	\begin{align*}
	& = \lim_{f \rightarrow \infty} \dfrac{\mathlarger{\int}_{\text{Bun}_G^{\mu_1<f/2+d/2-g+1}(k)}|H^0(\text{Sym}^2_0(\mathcal{E}) \otimes \mathcal{L}^\vee \otimes \mathcal{F})|\,\,d\mathcal{E}}{\mathcal{A}_{\mathcal{F}}(k)} . \lim_{f \rightarrow \infty} \frac{|H^0(\text{Sym}^2_0(\mathcal{E}) \otimes \mathcal{L}^\vee \otimes \mathcal{F})^{\text{reg}}|}{|H^0(\text{Sym}^2_0(\mathcal{E}) \otimes \mathcal{L}^\vee \otimes \mathcal{F})|} \\
	& \leq 4 \zeta_{C}(n+1) . \prod_{i=1}^n (\zeta_{C}(2i)) . \zeta_{C}(n+1)^{-1}.\prod_{i=1}^n \zeta_{C}(2i)^{-1}. \prod_{v \in |C|}(1 + |k_v|^{-n-1}|) \\ 
	& \hspace{3cm}(\text{by (\ref{equality 10}) and Proposition \ref{regular locus in general case} }) \\
	&= 4 . \prod_{v \in |C|}(1+|k_v|^{-n-1}). 
\end{align*}
\section{Proof of the main theorems} \label{section 5}
\subsection{Proof of theorem \ref{stack version of main theorem 1}}
The proof of this theorem is followed by the previous counting sections and the following remarks:
\begin{itemize}
\item When we restrict everything to the minimal locus, we need to modify two cases: in the case \ref{case 2}, the error function $f(q)$ picks up at most an extra factor $\zeta_{C}((n+2)(2n+1))$ due to Proposition \ref{regular locus in minimal case 1}, whether in the last case \ref{case 4}, instead of using Proposition \ref{regular locus in general case}, we use Proposition \ref{regular locus in minimal case 1} and \ref{regular locus in minimal case 2}, and obtain the same average number. 
\item If we consider the family of semi-stable hyperelliptic surfaces with irreducible fibers, then the error function $f(x)$ will pick up at most an extra factor $\prod_{v \in |C|}(1-\frac{d_v}{|k_v|^{2n+1}})^{-1}$ (see Proposition \ref{regular locus in semi-stable case}); hence, it still satisfies that $\lim_{q \rightarrow \infty} f(x) = 0.$ In the last case \ref{case 4}, we use Proposition \ref{regular locus in semi-stable case} instead of Proposition \ref{regular locus in general case} to obtain the exact number 4.
\end{itemize}
\subsection{Proof of theorem \ref{stack version of main theorem 2}}
We first need to know why we do not have any error function $f(q)$ in the transversal case. Here is the answer:
\begin{proposition} \label{ignore error term}
Let $\mathcal{E}$ be a  $G-$bundle. Assume that $\mathcal{E}$ has the canonical reduction to the Borel subgroup. Let $X_i$ be the associated line bundle as in the filtration (\ref{canonical reduction}) of $\mathcal{E}$ for $1 \leq i \leq n+1$, and denote $\mu_i$ the degree of $X_i$. If $ f/2 < \mu_i-\mu_{i+1} \leq f$ for all $1\leq i \leq n,$ $ f/2 < 2\mu_{n+1} - d \leq f$, and $2\mu_1-d < (2n+1)f,$ then discriminants of any sections of $(\mathcal{E} \times^G V^{\text{reg}})\otimes \mathcal{F}$ are not squarefree.  	
\end{proposition}
\begin{proof}
	As in the proof of Proposition \ref{factor through kostant 1}, any sections of $(\mathcal{E} \times^G V^{\text{reg}})\otimes \mathcal{F}$ are of the form:
	$$A=\begin{pmatrix}
	* & * & \cdots & * & * & *  \\
	x_1 & * & \cdots & * & * & *  \\
	0 & x_2 & \cdots & * & * & * \\
	\vdots & \vdots & \ddots & \vdots & \vdots & \vdots  \\
	0 & 0 & \cdots & x_2 & * & *  \\
	0 & 0 & \cdots & 0 & x_1 & * 
	\end{pmatrix}
	$$where $x_i \in H^0(C, X_i^{\vee} \otimes X_{i+1} \otimes \mathcal{F})$ for $1\leq i \leq n,$ and $x_{n+1} \in H^0(C, X_{n+1}^{\otimes -2} \otimes \mathcal{L} \otimes \mathcal{F}).$ By assumption, there exists one index $1\leq i\leq n+1$ such that $x_i(v) = 0$ for some $v \in |C|$ (since it is a section of a line bundle of positive degree). To estimate the order of $v$ in the discriminant of $A,$ we may assume that all entries of $A$ are in $\mathcal{O}_{K_v}.$ Now the determinant of $A$ can be expressed as
 \begin{equation*}
f_A(x) = \text{det}(A-xI) = \text{det}(B-xI_i)^2.f_1(x)+ x_i.\text{det}(B-xI_i).f_2(x) + x_i^2.f_3(x),
\end{equation*}
where $f_i(x) \in \mathscr{O}_{K_v}[x]$, and $B$ is a submatrix of $A$:
\[ B=\left( \begin{array}{cccccc}
* & * & \cdots & * & * & *  \\
x_1 & * & \cdots & * & * & *  \\
0 & x_2 & \cdots & * & * & * \\
\vdots & \vdots & \ddots & \vdots & \vdots & \vdots  \\
0 & 0 & \cdots & x_{i-2} & * & *  \\
0 & 0 & \cdots & 0 & x_{i-1} & * 
\end{array} \right). \]
By changing the variable $x$ to $x + \alpha$ for any root $\alpha \in \overline{K_v}$ of $\text{det}(B-xI_i)$ and working with an extended valuation in $K'=K_v(\alpha)$ of $\text{val}_v$, we further assume that 
$$f_A(x) =x^2.g_1(x)+ x.x_i.g_2(x) + x_i^2.f_3(x),$$ for some polynomial $g_i \in K'[x].$  Since $\text{det}(B-xI_i) \in \mathcal{O}_{K'}[x],$ we imply that $g \in \mathcal{O}_{K'}[x].$ Using the above form of $f_A$, it is easy to obtain that the resultant of $f_A$ and $f_A',$ and hence $\Delta(A),$ vanishes of order at least 2 at $v$. 
\end{proof}
Now we repeat the argument in the previous subsection to obtain the proof of Theorem \ref{stack version of main theorem 2} as follows: in the last case \ref{case 4}, we use Proposition \ref{regular locus in transversal case} instead of Proposition \ref{regular locus in general case}, and the average number in this subcase is $4.$ By Proposition \ref{ignore error term}, in the transversal case, we get an actual limit in contrast to the limsup in the general case. Hence, the following result is a direct corollary of Proposition \ref{regular locus in transversal case} and Theorem \ref{stack version of main theorem 2}:
\begin{corollary}
	We have the following lower bound of the average size of 2-Selmer groups:
	$$ \liminf_{d \rightarrow \infty} \frac{\sum_{H \in \mathcal{A}_{\leq d}}\frac{|Sel_2(H)|}{|\text{Aut}(H,\infty)|}}{\sum_{H \in \mathcal{A}_{\leq d}}\frac{1}{|\text{Aut}(H,\infty)|}} \geq 4. \prod_{v \in |C|} (1-\alpha_v) + 2,
	$$ where $\alpha_v$ is defined in Proposition \ref{regular locus in transversal case}, and note that $\lim_{q \rightarrow \infty} \prod_{v \in |C|} (1-\alpha_v) =1.$
\end{corollary}
\newpage

\bibliographystyle{alpha}
\bibliography{counting}

\begin{thebibliography}{HLHN14}

\bibitem[ABZ07]{localdensity}
Avner Ash, Jos Brakenhoff, and Theodore Zarrabi.
\newblock Equality of polynomial and field discriminants.
\newblock {\em Experimental Mathematics}, 16(3):367--374, 2007.

\bibitem[AT16]{Tho16}
Jack A.~Thorne.
\newblock On the average number of 2-selmer elements of elliptic curves over
  $\mathbb{F}_q(x)$ with two marked points.
\newblock 2016.
\newblock arXiv:1607.00997.

\bibitem[BG13]{BG13}
Manjul Bhargava and Benedict~H. Gross.
\newblock The average size of the 2-{S}elmer group of {J}acobians of
  hyperelliptic curves having a rational {W}eierstrass point.
\newblock In {\em Automorphic representations and {$L$}-functions}, volume~22
  of {\em Tata Inst. Fundam. Res. Stud. Math.}, pages 23--91. Tata Inst. Fund.
  Res., Mumbai, 2013.

\bibitem[BH04]{BH04}
Indranil Biswas and Yogish~I. Holla.
\newblock Harder-{N}arasimhan reduction of a principal bundle.
\newblock {\em Nagoya Math. J.}, 174:201--223, 2004.

\bibitem[BS13a]{BS13c}
M.~Bhargava and A.~Shankar.
\newblock The average number of elements in the 4-selmer groups of elliptic
  curves is 7.
\newblock 2013.
\newblock arXiv:1312.7333.

\bibitem[BS13b]{BS13d}
M.~Bhargava and A.~Shankar.
\newblock The average size of the 5-selmer group of elliptic curves is 6, and
  the average rank is less than 1.
\newblock 2013.
\newblock arXiv:1312.7859.

\bibitem[BS15a]{BS13a}
Manjul Bhargava and Arul Shankar.
\newblock Binary quartic forms having bounded invariants, and the boundedness
  of the average rank of elliptic curves.
\newblock {\em Ann. of Math. (2)}, 181(1):191--242, 2015.

\bibitem[BS15b]{BS13b}
Manjul Bhargava and Arul Shankar.
\newblock Ternary cubic forms having bounded invariants, and the existence of a
  positive proportion of elliptic curves having rank 0.
\newblock {\em Ann. of Math. (2)}, 181(2):587--621, 2015.

\bibitem[BSW16]{bhargavasquarefree}
Manjul Bhargava, Arul Shankar, and Xiaoheng Wang.
\newblock Squarefree values of polynomial discriminants i.
\newblock 11 2016.
\newblock https://arxiv.org/abs/1611.09806.

\bibitem[HLHN14]{HLN14}
Q.~P. Ho, V.~B. L\^e~H\`ung, and B.~C. Ng\^o.
\newblock Average size of 2-{S}elmer groups of elliptic curves over function
  fields.
\newblock {\em Math. Res. Lett.}, 21(6):1305--1339, 2014.

\bibitem[Lan18]{aaron}
Aaron Landesman.
\newblock The geometric average size of selmer groups over function fields.
\newblock 2018.
\newblock preprint.

\bibitem[Lev07]{Lev08}
Paul Levy.
\newblock Involutions of reductive {L}ie algebras in positive characteristic.
\newblock {\em Adv. Math.}, 210(2):505--559, 2007.

\bibitem[Lev09]{Lev09}
Paul Levy.
\newblock Vinberg's {$\theta$}-groups in positive characteristic and
  {K}ostant-{W}eierstrass slices.
\newblock {\em Transform. Groups}, 14(2):417--461, 2009.

\bibitem[Liu02]{Liu06}
Qing Liu.
\newblock {\em Algebraic geometry and arithmetic curves}, volume~6 of {\em
  Oxford Graduate Texts in Mathematics}.
\newblock Oxford University Press, Oxford, 2002.
\newblock Translated from the French by Reinie Ern\'e, Oxford Science
  Publications.

\bibitem[Poo03]{Poo03}
Bjorn Poonen.
\newblock Squarefree values of multivariable polynomials.
\newblock {\em Duke Math. J.}, 118(2):353--373, 2003.

\bibitem[SBR90]{Ray90}
W.~Lutkebohmert S.~Bosch and M.~Raynaud.
\newblock {\em N\'{e}ron Model}, volume~21 of {\em Ergebnisse der Mathematik
  und ihrer Grenzgebiete,}.
\newblock Springer-Verlag, 1990.

\bibitem[SW18]{SW13}
Arul Shankar and Xiaoheng Wang.
\newblock Rational points on hyperelliptic curves having a marked
  non-{W}eierstrass point.
\newblock {\em Compos. Math.}, 154(1):188--222, 2018.

\bibitem[Thi]{DVT17}
Dao~Van Thinh.
\newblock Average size of 2- selmer groups of hyperelliptic curves over
  function fields.
\newblock preprint.

\bibitem[Wan13]{Wan1}
Xiaoheng Wang.
\newblock {\em Pencils of quadrics and {J}acobians of hyperelliptic curves}.
\newblock ProQuest LLC, Ann Arbor, MI, 2013.
\newblock Thesis (Ph.D.)--Harvard University.

\end{thebibliography}

\end{document}